\documentclass[12pt]{article}
\usepackage[utf8]{inputenc}
\usepackage{graphicx}
\usepackage{float}
\usepackage{color}
\usepackage{times}
\usepackage[hidelinks]{hyperref}
\usepackage{enumerate,latexsym}
\usepackage{latexsym}
\usepackage{amsmath,amssymb}
\usepackage{graphicx}
\usepackage{amsmath,amsthm,amsfonts,amssymb}
\usepackage{graphicx}
\usepackage{dsfont}
\usepackage{fullpage}
\usepackage{pdfsync}

\newcommand{\abs}[1]{\vert #1\vert}
\renewcommand{\a}{\alpha}
\newcommand{\bbX}{\mathbb{X}}
\newcommand{\bbT}{\mathbb{T}}
\newcommand{\bbK}{\mathbb{K}}
\newcommand{\bbD}{\mathbb{D}}
\newcommand{\bbE}{\mathbb{E}}
\newcommand{\g}{\gamma}
\newcommand{\G}{\Gamma}
\newcommand{\wt}[1]{\widetilde{#1}}
\newcommand{\wh}[1]{\widehat{#1}}
\newcommand{\ol}[1]{\overline{#1}}
\newcommand{\cB}{\mathcal{B}}
\newcommand{\Q}{\mathds{Q}}
\renewcommand{\S}{\Sigma}

\newcommand{\sn}[1]{{\mathds{S}^{#1}}}
\newcommand{\hn}[1]{{\mathds{H}^{#1}}}

\newcommand{\cC}{\mathcal{C}}

\newcommand{\N}{\mathbb{N}}

\newcommand{\Hy}{\mathbb{H}}
\newcommand{\Ma}{\mathcal{M}}

\newcommand{\Oc}{\mathcal{O}}
\newcommand{\Z}{\mathbb{Z}}
\newcommand{\D}{\mathbb{D}}
\newcommand{\Sph}{\mathds{S}}

\newtheorem{theorem}{Theorem}[section]
\newtheorem{corollary}[theorem]{Corollary}

\newtheorem{proposition}[theorem]{Proposition}
\newtheorem{conjecture}[theorem]{Conjecture}

\newtheorem{case}{Case}

\theoremstyle{definition}
\newtheorem{definition}[theorem]{Definition}
\newtheorem{remark}[theorem]{Remark}
\newtheorem{example}[theorem]{Example}

\title{Geometrization in Geometry}

\author{I. de Freitas \and A. Ramos}

\newcommand\R{\mathds{R}}

\begin{document}

\maketitle

\begin{center}
\noindent
\begin{minipage}{0.85\textwidth}\parindent=15.5pt

\smallskip
\begin{center}
\large{\textsl{Dedicated to Professor Renato Tribuzy\\ on the occasion of his 75th birthday}}
\end{center}

\smallskip

{\small{
\noindent {\bf Abstract.} 
So far, the most magnificent breakthrough in mathematics
in the 21st century is the Geometrization Theorem, a bold conjecture
by William Thurston (generalizing Poincaré's Conjecture) and proved
by Grigory Perelman, based on the program suggested by Richard Hamilton. In this
survey article, we will explain the statement of this result, 
also presenting
some examples of how it can be used to obtain interesting 
results in differential geometry.}
\smallskip

\noindent {\bf{Keywords:}} geometrization, hyperbolization, Poincaré's 
conjecture, orbifolds, Seifert fibered spaces, geometric topology.
\smallskip

\noindent{\bf{2020 Mathematics Subject Classification:}} Primary: 51H25. 
Secondary: 53A10, 57K32, 53C42.
}

\end{minipage}
\end{center}

\section{Introduction} 

In certain sense, a {\em mathematical object} is an idealistic
representation of the reality, and in several contexts this
idealization suffices to give us insight about nature and,
more broadly, the universe. However, in the past few centuries,
mathematics have developed to a degree in which it escapes 
reality and immerses into the world of imagination. Astonishing
is the fact that, even in its retreat from reality, mathematics
has the ability to provide tools that later can be used to 
explain natural effects and develop science and technology as a 
whole.

In this survey paper, dedicated to Professor Renato Tribuzy,
we are interested in the shapes of certain
mathematical objects called {\em 3-manifolds}.
More specifically, we present what is arguably the most 
significant development on mathematics in the 21st century, 
which is the {\em geometrization theorem}. Rather than
presenting its proof, which was given in a series of 
papers~\cite{per1,per2,per3}
by Grigory Perelman, we will focus on its precise 
statement and on some of the background material necessary 
for its understanding. 

Although the main topic of this article is topology of 
3-manifolds, it is actually a paper in {\em geometric topology},
but we would like to mention that it was written with a special care
for differential geometers readers, which sometimes may have trouble 
(as we did) for translating the concepts of topology to a more
geometric perspective.
We also notice that the main results concerning 
basic aspects of geometrization of 3-manifolds that
we present may be found in a deeper level in the books
of M. Aschenbrenner, S. Friedl, H. Wilton~\cite{3mg}, B. 
Martelli~\cite{martelli} or P. Scott~\cite{scott}.

Next, we explain the organization of the manuscript. In
Section~\ref{sectandg}, we present the intuitive 
concept of topology and how it 
is related to geometry. We also explain the concept of an orbifold,
with emphasis for dimension 2, which is necessary 
for geometrization.
Section~\ref{sec3manf} is where we present the geometrization theorem,
after making precise the definition and classification of 
Seifert fibered spaces. In Section~\ref{sechyperb}, we 
particularize to the hyperbolic geometry, presenting Thurston's 
hyperbolization criterion, several examples and some applications.

\section*{Acknowledgments}

The authors would like to thank 
Colin Adams and Vanderson Lima for useful discussions in the
preparation of this manuscript and the referees for several suggestions 
and comments that improved the paper.
The first author was partially 
supported by CAPES/Brazil, and second author was partially
supported by CNPq/Brazil, grant number 312101/2020-1.

\section{Topology and Geometry}\label{sectandg}

In an intuitive manner, we may say that
{\em Topology} is the study of shapes. If two topological spaces
can be obtained one from the other by means of a homeomorphism,
then Topology regards both objects as the same.
On the other hand, {\em Riemannian Geometry} is the study of
certain smooth objects, called manifolds, together
with their intrinsic distances, measured by (Riemannian) metrics,
and two Riemannian manifolds will be thought as being the same
if there is an isometry (which is a metric preserving
diffeomorphism) between them, see Figure~\ref{figspheres}.

At first glance these two concepts may seem to be
disjoint, but there are several connections between them, which 
make both theories richer. First of all, 
every differentiable manifold $M$
has an intrinsic topology associated to it\footnote{This intrinsic
topology is obtained by stating that a subset $U\subset M$ is
open if $x_\a^{-1}(U\cap x_\a(U_\a))$ is an open set of 
$\R^n$ for every local chart $x_\a\colon U_\a\subset \R^n\to M$. 
It is important to mention that 
we assume that this topology
is {\em Hausdorff}, i.e., any two points can be separated by disjoint
neighborhoods, and that $M$ can be covered by a countable number 
of charts. When $M$ is connected, these assumptions are equivalent
to the existence of a differentiable
partition of unity on $M$, which is an essential tool for the study
of several questions on manifolds.}, 
hence we are allowed to
regard $M$
as a topological space and to use
the expression {\em the topology of $M$} to refer to such topological
space structure.
For details, see~\cite[Chapter~2]{Bredon} (which contains a topological
introduction) or~\cite[Chapter~0]{manfr} (for a more geometric 
point of view).

Although
there are several topological spaces that are not manifolds
and even homeomorphic manifolds that are not 
diffeomorphic~\cite{milnorsphere},
a striking observation is that it is possible to relate
geometry and topology in elegant and deep manners. For
instance, if $S$ is a closed (compact, without boundary) 
surface endowed with a Riemannian
metric, then Gauss-Bonnet Theorem implies that its 
total Gaussian
curvature $\int_SK_S$ (a geometric quantity) 
is related with its Euler characteristic (a topological 
invariant)
$\chi(S)$ by
$$\int_S K_S = 2\pi \chi(S).$$
In particular, independently on the metric considered
on the 2-sphere (for instance, see
the spheres depicted in Figure~\ref{figspheres}), its total
curvature will always equals $4\pi$. We cannot {\em bend a 
sphere} to increase or decrease its total curvature 
without tearing it apart, thus changing its topology.

\begin{figure}[t]
\centering
\includegraphics[width=\textwidth]{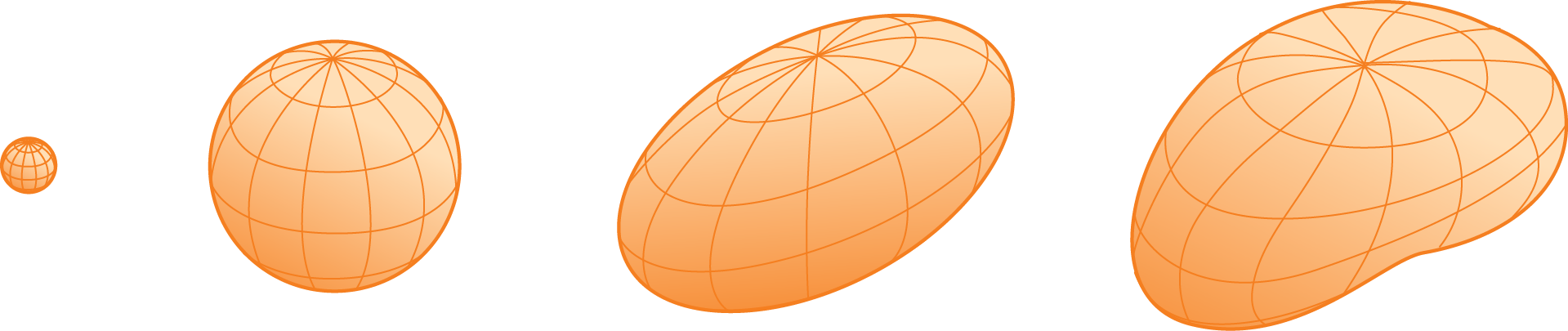}
\caption{\label{figspheres}
As topological spaces, all these 2-spheres are equivalent, since
they are homeomorphic.
However, from a geometric point of view, they are distinguishable
from each other by their respective metrics.}
\end{figure}

Other interesting results relating topology and geometry are
the Bonnet-Myers Theorem and the Cartan-Hadamard Theorem, which
restrict the topology of a given geometry:

\begin{theorem}[{Bonnet-Myers~\cite{Bonnet,myers}}]
Let $M$ be a complete Riemannian manifold with sectional curvature
$K_M\geq \delta>0$. Then, $M$ is compact and $\pi_1(M)$ is 
finite.
\end{theorem}

\begin{theorem}[{Cartan-Hadamard, see~\cite{helgason} or~\cite{manfr}}]
Let $M$ be a complete Riemannian manifold with sectional curvature
$K_M\leq 0$. Then, the exponential map of $M$ at any point $p$
is a covering transformation. In particular,
the universal covering of $M$ is
diffeomorphic to $\R^n$, where $n = {\rm dim}(M)$.
\end{theorem}

But perhaps even more interestingly than the aforementioned
results, geometry can be used to 
answer an old, natural topology question: 
{\em what are all the possible
shapes of surfaces?} This is going to be discussed
in the next section.

\subsection{Surfaces.}

In this section, we will discuss what are 
the possible topologies for
a closed, orientable surface, and we will also present
the geometrization\footnote{By geometrization of a manifold $M$
we mean finding a decomposition of $M$ such that each component admits
a geometric structure, which is a complete metric locally isometric 
to a given simply connected homogeneous manifold; a detailed description
will be given in Section~\ref{secgeomm}.} theorem for surfaces, 
which will be a starting
point for the reader to understand the geometrization of a
closed, orientable 3-manifold. 

\begin{definition}
Let $S$ be a closed, orientable surface. Then, the 
{\em genus} of $S$ is the maximal number of
pairwise disjoint simple closed curves in $S$ whose collection do not 
separate $S$.
\end{definition}

\begin{example}
Since the sphere $\sn2$ 
is simply connected, every simple closed curve in $\sn2$
separates. Thus, its genus is zero. 

\noindent 
The genus of the torus $\bbT^2 = \sn1\times\sn1$ is positive, since
there exists a nonseparating simple closed curve 
(for instance, $\sn1\times\{p\}$ for any $p\in \sn1$) in $\bbT^2$.
However, if $\g_1$ and $\g_2$ are pairwise disjoint simple closed curves
that individually do not separate,
then $\bbT^2\setminus \g_1$ has the topology of $\sn1\times(0,1)$,
and we may see $\g_2$ as a
nontrivial simple closed curve in $\sn1\times(0,1)$, so $\g_2$
is parallel to $\sn1\times\{1/2\}$, thus separating. This
proves that the collection $\g_1\cup \g_2$ separates $\bbT^2$,
showing that its genus is equal to one.
\end{example}

More than an interesting topological invariant, the genus of
a closed, orientable surface is sufficient to completely 
classify its topology. This follows from the
following {\em Classification
Theorem}, whose original proof dates back from the 1860's,
and a modern approach presented by J. Conway as the 
{\em Zero Irrelevancy Proof} (or ZIP proof, for short) 
can be found in~\cite{franweeks}.

\begin{theorem}[Classification Theorem]
Let $S$ be a connected, orientable, closed surface. 
Then, $S$ is diffeomorphic to a 
connected sum (see Definition~\ref{consum} below)
of the 2-sphere
$\sn2$ with $g\geq 0$ copies of the 2-torus $\bbT^2$.
\end{theorem}

We notice that the statement of the Classification Theorem
is more general than the presented above, which was a 
restriction to the orientable case. We also used the 
concept of connected sum, depicted
in Figure~\ref{figconnectedsum} and
whose definition 
we present next in dimensions 2 and 3.

\begin{definition}\label{consum}
Let $M_1$ and $M_2$ be two oriented
manifolds of the same dimension $n\in\{2,3\}$\footnote{The definition
of connected sum of $M_1$ and $M_2$ can be generalized for 
any $n>3$. However, in dimensions
higher than 3,
the differentiable structure of the connected sum $M_1\# M_2$ can depend 
not only on the chosen orientations on $M_1$ and $M_2$ but also
on the choice of the gluing map $f$, while if $n\in \{2,3\}$ this
does not occur.}. 
Then, the {\em connected sum}
of $M_1$ and $M_2$ is the manifold $M_1\# M_2$ obtained
by the following operation: let $B_1,\,B_2$ be respective 
$n-$balls in $M_1,\,M_2$ with boundaries $S_1,\,S_2$.
Let $f$ be an orientation-reversing 
diffeomorphism from $S_1$ to $S_2$. Then
$M_1\# M_2 = 
(M_1\setminus B_1)\cup (M_2\setminus B_2)/\sim$,
where we identify $S_1\ni x \sim f(x) \in S_2$.
\end{definition}

As previously explained, the Classification Theorem
allows us to list all closed, orientable surfaces in terms
of their genus, since such
a surface has genus $g$ if and only if
it is the connected sum of $\sn2$ with $g$ tori. Having this 
construction in mind, we may now present the Uniformization Theorem,
which was
conjectured by Felix Klein and Henri Poincaré in the 1880's 
and proved independently by Poincaré~\cite{poinc} and 
Koebe~\cite{koebe} in 1907.

\begin{figure} 
\centering
	\includegraphics[scale=0.38]{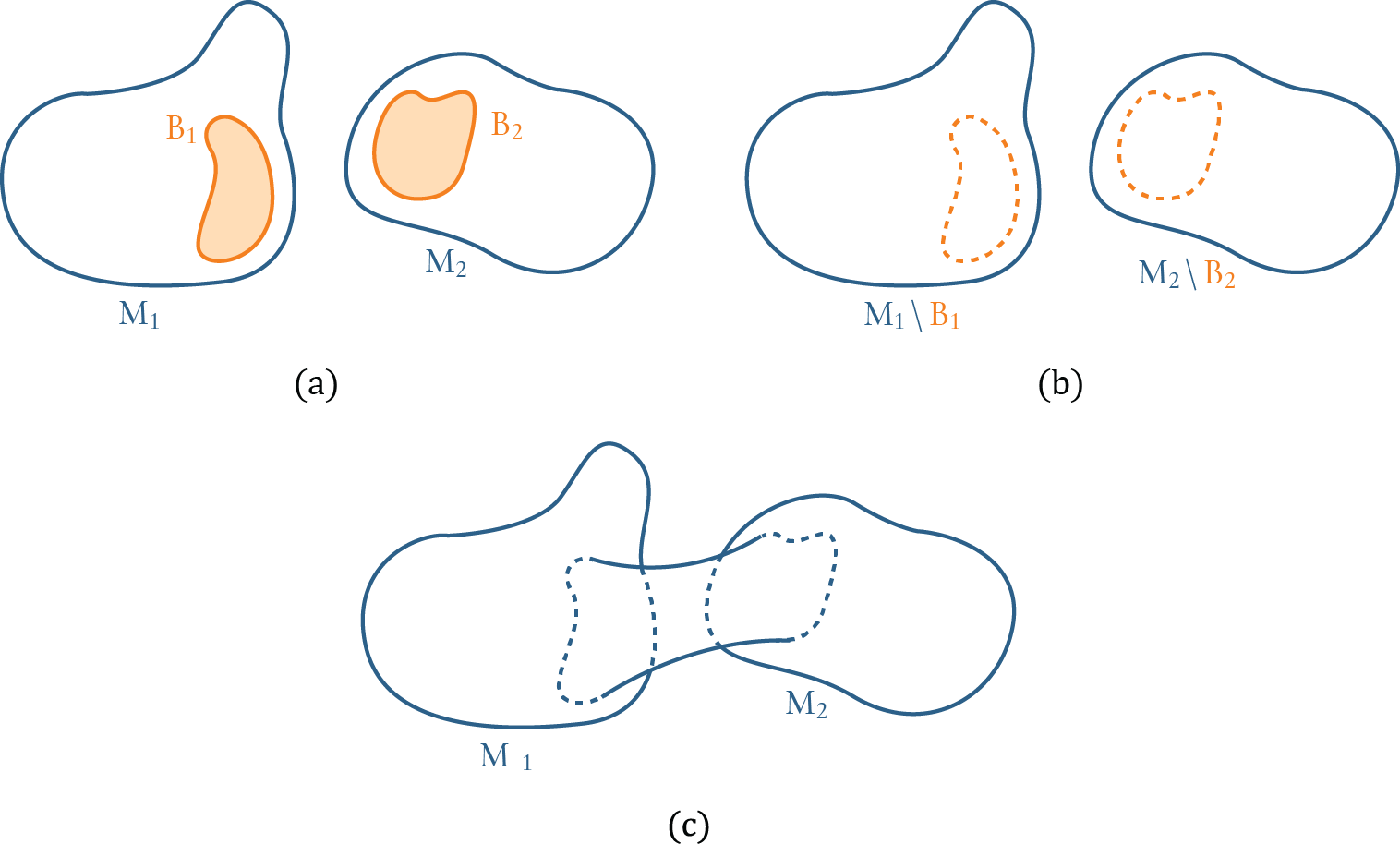}
\caption{The connected sum of $M_1$ and $M_2$, as in 
Definition~\ref{consum},
removes an $n$-ball $B_1$ from $M_1$ and an $n$-ball $B_2$, and then
joins the resulting manifolds with spherical 
boundary by an orientation-reversing
diffeomorphism, which has the visual effect of gluing them by a
cylindrical neck.
\label{figconnectedsum}}
\end{figure}

\begin{theorem}[Uniformization theorem]
\label{geomSf}
Let $S$ be an orientable, closed surface. Then, 
$S$ admits a metric of constant curvature $k$. Furthermore, 
if $g$ is the genus of $S$, the following hold:
\begin{enumerate}
\item $k >0$ if and only if $g=0$.
\item $k =0$ if and only if $g=1$.
\item $k <0$ if and only if $g\geq2$.
\end{enumerate}
\end{theorem}

In Section~\ref{secgeomm} (more precisely in Theorem~\ref{geomSf2}),
we will revisit 
Theorem~\ref{geomSf}, justifying the fact that
it is also known by the name
of {\em geometrization for surfaces}.
At the present moment, we will restrain ourselves to observe
that the uniformization theorem shows that a connected, closed 
and orientable surface has a special metric of constant
curvature, and its Riemannian universal covering is isometric
(after a homothety) to one and only one of
the space forms $\sn2,\,\R^2,\,\hn2$, which will be called
the {\em model geometry} for $S$. 
In particular, if $X$ is the model geometry for
a surface $S$ and $G = {\rm ISO}(X)$, there exists a subgroup
$H$ of $G$ such that $S$ is diffeomorphic to
the quotient $X/H$. We also point out that the original statement
of the Uniformization Theorem (which is well-known to be equivalent
to the one presented above) makes use of the theory of Riemann surfaces,
stating that that any simply connected Riemann surface is either conformal
to the open unit disk, to the complex plane or to the Riemann sphere.

\begin{example}
Let $S = \bbT^2$ be the torus. Then, the model geometry for
$S$ is $\R^2$ and if we let $H$ be the group generated by
two linearly independent translations in $\R^2$, it follows
that $S = \R^2/H$.
\end{example}

Generalizing the concept of surfaces, we will next introduce
the concept of {\em orbifolds}, which are certain not-so-well 
behaved quotients of manifolds. Orbifolds play a major role
in the geometrization of 3-manifolds, and this nomenclature
was introduced in the 1970's by William Thurston, after a 
vote by his students, but they appeared earlier
in the literature under the
name of V-manifolds.

\subsection{Orbifolds}
In some sense, orbifolds are structures used to understand group actions over manifolds, but differently from surfaces, where the group action is properly discontinuous\footnote{The action of a group $G$ over a manifold $X$ is properly discontinuous if for
every compact $K\subset X$ the number of elements $g$ of $G$ that satisfy $g K \cap K \neq \emptyset$ is finite.} and free\footnote{The action of a group $G$ over a set $X$ is called {\em free} if it has no fixed points unless it is the action of the identity element.}, on orbifolds the groups will act uniquely in a properly discontinuous manner. We also recall that the quotient of a topological space $X$ by a group $G$, denoted by $X/G$, is the set of orbits, together with the quotient topology.

Before giving the precise definition of an orbifold, we will 
present some examples to bring up some intuition to the reader.

\begin{example}\label{billiardeg}
Let $R \subset \R^2$ be a rectangle and $G$ the group
of isometries of $\R^2$ generated by the reflections
along the four lines containing the sides of $R$.
\begin{figure}[H] 	
\centering
	\includegraphics[scale=0.5]{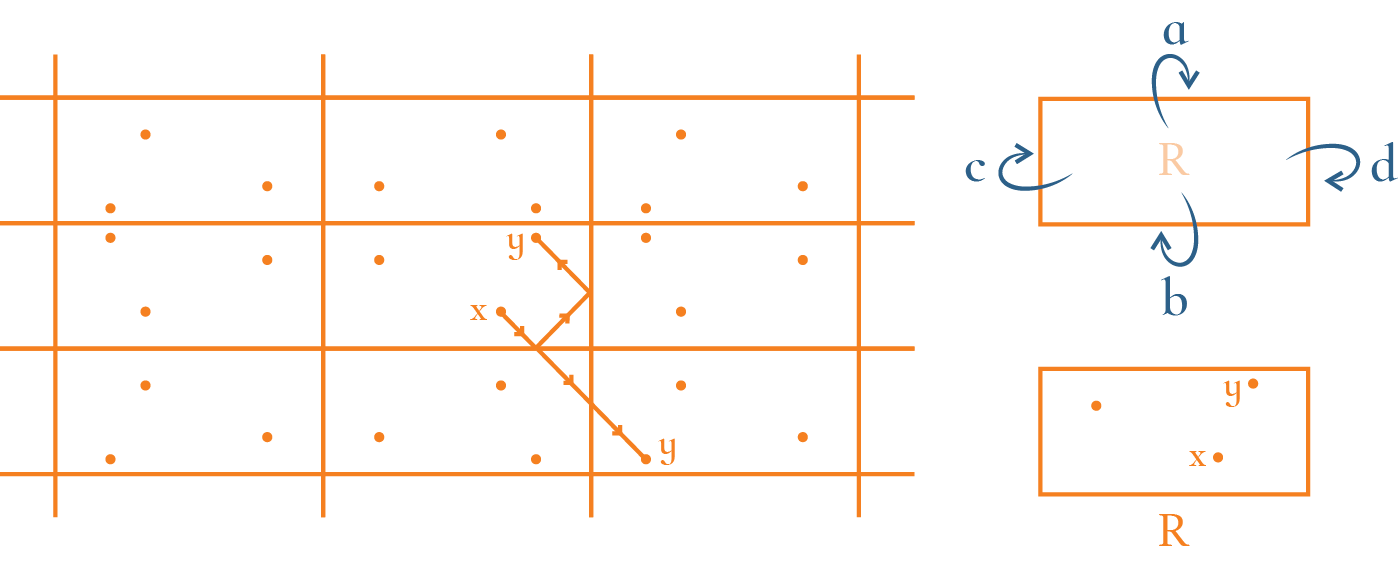}
\caption{Tiling of $\R^{2}$ generated by the reflections
$a,b,c,d$ over the sides of $R$.\label{img2.5}}
\end{figure}

The reflections $a,b$ (and also $c,d$) over two parallel
sides give rise to the free product group
$D_{\infty} = \Z_{2} \ast \Z_{2}$. 
Thus, in this case we have $G=D_{\infty} \times D_{\infty}$
and the quotient space $\R^{2} / G$ is the rectangle $R$.

This space is called the {\em rectangular billiard}, because
if we think the points of $R$ as pool balls, to hit
a certain point $y$ from another point $x$, it suffices
to aim in any other copy of $y$ into a reflected image of $R$,
as shown in Figure~\ref{img2.5}.
\end{example}

\begin{example}\label{pillowcaseeg}
Let $G = D_\infty\times D_\infty$ be the group as in 
Example~\ref{billiardeg} and let $H$ be the index-2 subgroup
of $G$ given by the orientation preserving isometries.
Then, $\R^{2} / H$ is the {\em pillowcase space},
which is topologically a 2-sphere with 4 singularities,
at the points corresponding to the vertexes of $R$ (see
Figure~\ref{figpillowcase}).

\begin{figure}[H] 	
\centering
	\includegraphics[scale=0.48]{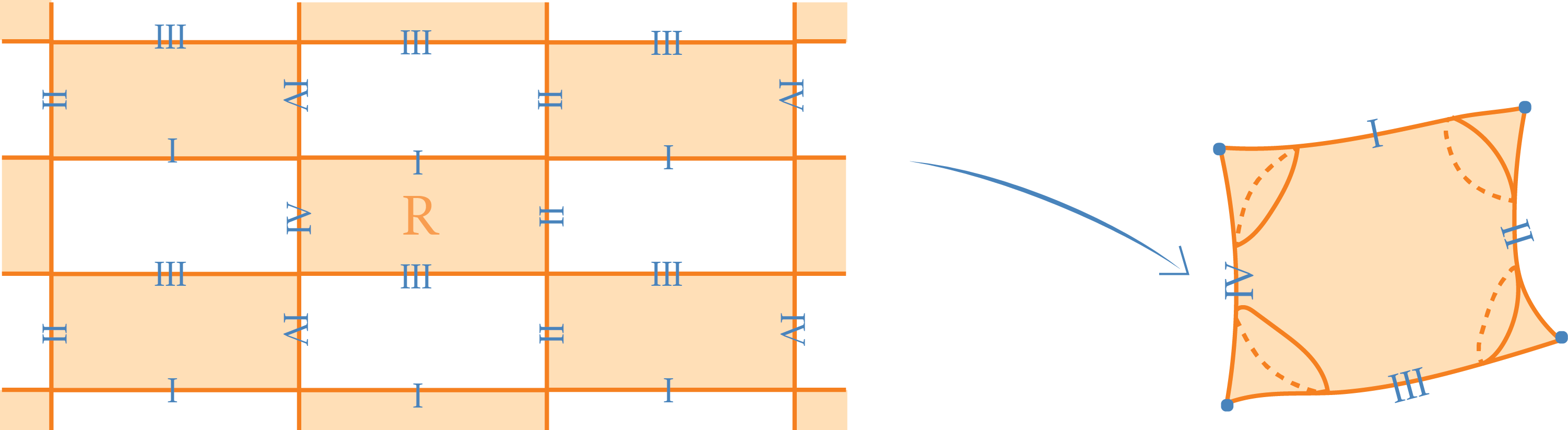}
\caption{Quotient of $\R^{2}$ by $H$. Note that $H$
maps the filled rectangles in the tiling 
into other filled rectangles with the
depicted orientation.\label{figpillowcase}}
\end{figure}
\end{example}

Taking into consideration the two examples above,
we say that an \textbf{orbifold} $\Oc$ 
is a topological space locally modelled in the quotient
of $\R^n$ by finite group actions with some
additional structure. We will first present the rigorous
definition given by Thurston~\cite{thur},
but we notice that the orbifolds that appear 
in order to understand geometrization
of 3-manifolds are simpler, and most of the times
the intuition behind the definition will suffice.

\begin{definition}[Thurston]
An $n$-dimensional orbifold
$\Oc$ is a Hausdorff space $X _{\Oc}$
(called the {\em base space}) endowed
with the following additional structure. There exists
a covering of 
$X_{\Oc}$ by a collection of open sets $\{ U_{i} \}$ that 
satisfy:
\begin{itemize}
\item The collection $\{U_i\}$ is closed with respect to 
finite intersections;
\item For each $U_i$ there is a finite group
$\G_i$ together with an action
of $\G_i$ over an open set $\wt{U}_i\subset \R^n$ and a 
homeomorphism
$\varphi_i\colon U_i \to \wt{U}_i/\G_i$;
\item If $U_i\subset U_j$ for some $i,j$, there exists
an injective homomorphism $f_{ij}\colon
\G_i \hookrightarrow \G_j$ and an embedding
$\wt{\varphi}_{ij}:\wt{U}_{i} 
\hookrightarrow \wt{U}_{j}$ that is
equivariant with respect to $f_{ij}$, in the sense that
$\wt{\varphi}_{ij}(\gamma x) 
= f_{ij}(\gamma)\wt{\varphi}_{ij}(x)$
for all $\g\in \G_i$, $x \in U_i$, so the diagram below commutes
\begin{center}
\includegraphics[width=0.4\textwidth]{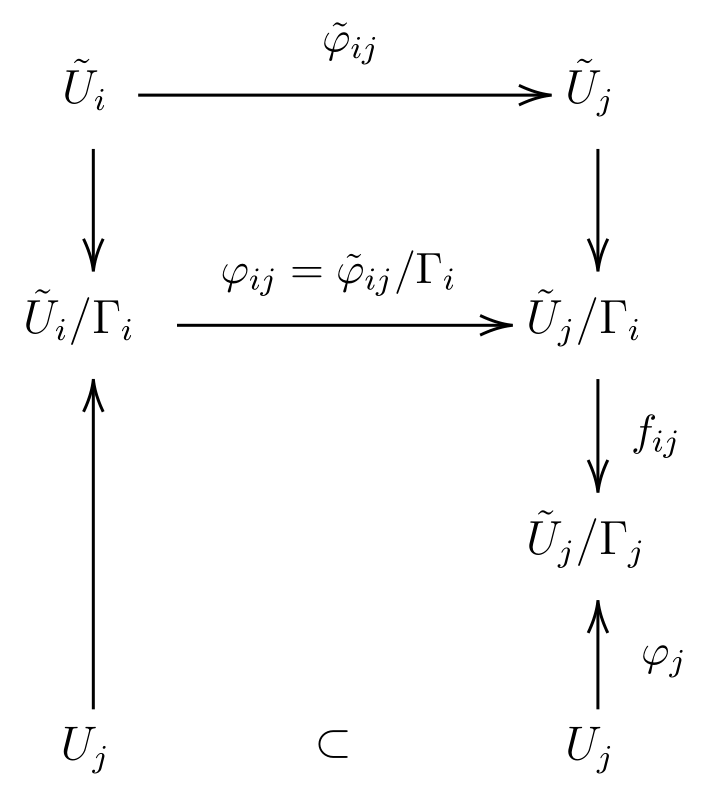}
\end{center}
\end{itemize}
We regard $\tilde{\varphi}_{ij}$ as being defined only up to composition with elements of $\Gamma_{j}$, and $f_{ij}$ as being defined up to conjugation by elements of $\Gamma_{j}$. It is not generally true that $\tilde{\varphi}_{ik} = \tilde{\varphi}_{jk} \circ \tilde{\varphi}_{ij}$ when $U_{i} \subset U_{j} \subset U_{k}$, but there should exist an element $\gamma \in \Gamma_{k}$ such that $\gamma \tilde{\varphi}_{ik} = \tilde{\varphi}_{jk} \circ \tilde{\varphi}_{ij}$ and $\gamma \cdot f_{ik}(g) \cdot \gamma^{-1} = f_{jk} \circ f_{ij}(g)$. 
\end{definition}

If $\Oc$ is an orbifold and $x\in \Oc$, we let 
$\G_x$ denote the {\em isotropy group of the point $x$}, i.e.,
$$\G_x = \{g\in \G_i\mid x\in U_i\text{ and }gx = x\},$$
which is well-defined up to conjugation.
In other words, in
a neighborhood $U_i = \wt{U}_i/\G_i$ of $x$,
$\G_x$ is the subgroup of $\G_i$ that acts over $\wt{U}_i$
leaving $\wt{x}$ invariant, where $\wt{x}\in \wt{U}_i$
projects over $x$.

\begin{definition}
The {\em singular locus} of an orbifold $\Oc$ is 
the set 
$\Sigma_{\Oc} = \{ x\in \Oc \mid \Gamma_{x} \neq \{1\}\}$,
which means that
$\Sigma_{\Oc}$ is the subset of $\Oc$ for which there
exists a correspondent neighborhood $\wt{U}$ with
nontrivial isotropy group.
If $p\in \S_{\Oc}$, we say that $p$ is a {\em singular} point.
Otherwise, $p$ is called {\em regular}. 
Note that the singular locus is a closed subset of $\Oc$.
\end{definition}

In this article, we are interested in 2-dimensional orbifolds,
where the {\em singular locus} is simpler than in the general
situation, admitting a simple description as we next present.

\begin{proposition}[Thurston~\cite{thur}] \label{locussin}
The {\em singular locus} of a 2-dimensional orbifold
has one of the three local models below:
\begin{enumerate}
\item {\bf Lines of reflection}. $\R^{2}/ \Z_{2}$, where 
$\Z_{2}$ acts by reflection along a line in $\R^2$;
\item {\bf Elliptic points of order $n\geq2$}.
$\R^{2}/ \Z_{n}$, where $\Z_{n}$ acts on $\R^2$ by a rotation
of $2\pi/n$;
\item {\bf Corner reflectors of order $n$}.
$\R^{2}/ D_{n}$, where $D_{n}$ is the dihedral group 
of order $2n$, with representation
$\langle a,b \ : \ a^{2} = b^{2} = (ab)^{n} = 1 \rangle $, 
and the generators $a$ and $b$ correspond to reflections
along lines intersecting themselves at an angle $\pi/n$.
\end{enumerate}
\end{proposition}

The proof of the proposition uses that
the only finite subgroups of the orthogonal group
${\rm O}(2)$ are the ones described above, together with the
fact that
given a local coordinate system $U = \wt{U}/\G$ to an orbifold
$\Oc$, there exists a homeomorphism between a neighborhood
of $U$ and a neighborhood of the origin in the orbifold
$\R^2/\G$, where $\G\subset {\rm O}(2)$ is a finite subgroup.
For more details, see Thurston~\cite[Proposition~13.3.1]{thur}

\begin{figure} 	
\centering
	\includegraphics[scale=0.22]{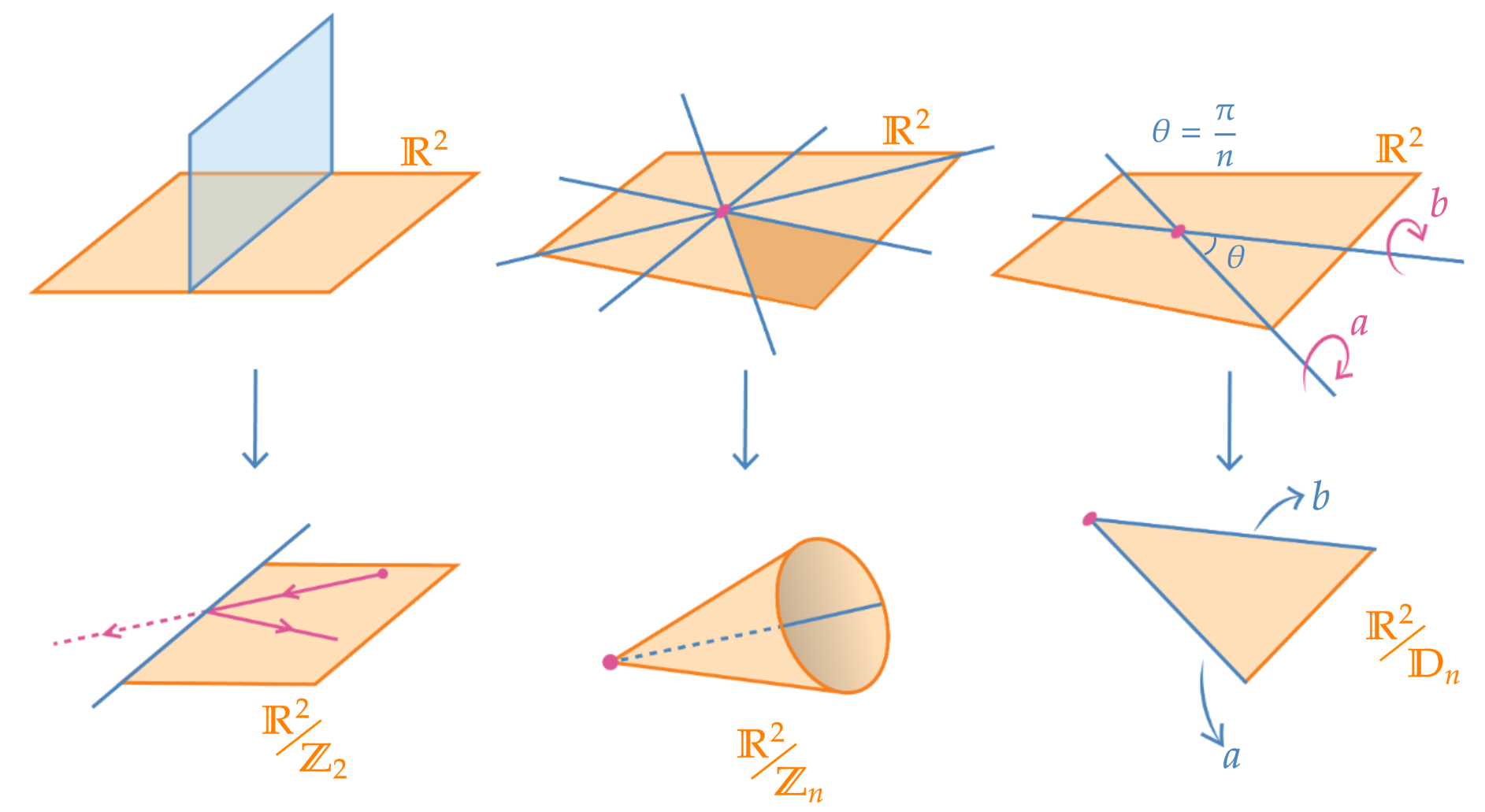}
\caption{Quotients of $\R^{2}$ by the groups $\Z_{2}, \Z_{n}$ and
$D_{n}$, giving the respective singular loci for a 2-dimensional
orbifold: a line of reflection, an elliptic point of order $n$ 
and a corner reflector of order $n$.}
\end{figure}

\begin{example}
Recall Example~\ref{billiardeg}, where $R$ is a rectangle in $\R^{2}$ and $G$ is the group of isometries generated by reflections $a,b,c,d$. Now that we have the definition of orbifold and singular locus, we may explore the example with more depth.

Points in $R$ can be inside the rectangle, or on the edge, or a vertex. It is not hard to notice that points inside the rectangle are not fixed by any element of $G$. Given $x \in int(R)$ and $U_{i}$ a neighborhood of $x$, there is a homeomorphism $\varphi_{i}: U_{i} \to \tilde{U}_{i}$, where $\tilde{U}_{i} \subset \R^{2}$.

If $x \in R$ is a point in the edge labeled as $d$, but not a vertex, for each neighborhood $U_{i}$ of $x$ we have $\varphi_{i}: U_{i} \to \tilde{U}_{i}/ \Gamma_{i}$, where $\tilde{U}_{i} \subset \R^{2}$. $\Gamma_{i}$ is the group generated by the reflection $d$ and $d^{2} = id$, therefore, $G \supset \Gamma_{i} \cong \Z_{2}$ and $d$ is a line of reflection.

If $x$ is the vertex of the edges $a$ and $d$, for each neighborhood $U_{i}$ of $x$ we have $\varphi_{i}: U_{i} \to \tilde{U}_{i}/ \Gamma_{i}$, where $\tilde{U}_{i} \subset \R^{2}$. $\Gamma_{i}$ is generated by $a$ and $d$ according to the relations $a^{2} = id$, $d^{2} = id$, $(ad)^{2} = id$. By definition $G \supset \Gamma_{i} \cong D_{2}$ and the vertex is a corner reflector. In fact, all vertices of $R$ are corner reflectors. 
\end{example}

\begin{example}
In Example~\ref{pillowcaseeg} we have $H \subset G $ the group of orientation-preserving isometries generated by reflections $a,b,c,d$ over the lines containing the four sides of $R$. It is straightforward to see that $d\circ a$, $a\circ c$, $c\circ b$ and $b\circ d$ are all elements of $H$, each of which leaves one of the four vertexes $P_1,\,P_2,\,P_3,\,P_4$ of the original rectangle $R$ invariant. Thus, in the pillowcase orbifold $\Oc = \R^2/H$, $\{P_1,\,P_2,\,P_3,\,P_4\}\subset \Sigma_{\Oc}$. In fact, these are all the singular points of this orbifold.

If $x$ is the vertex of the edges $a$ and $d$, for each neighborhood $U_{i}$ of $x$ we have $\varphi_{i}: U_{i} \to \tilde{U}_{i}/ \Gamma_{i}$, where $\tilde{U}_{i} \subset \R^{2}$. $\Gamma_{i}$ is generated by $d\circ a$, a rotation by $\pi$, therefore we have that $\Gamma_{i} \cong \Z_{2}$ and $x$ is an elliptic point of order $2$,
and the same holds for all vertices of $R$.
\end{example}

The next proposition is of great importance to the
study of orbifolds in the context of geometrization. It may
even be used as an intuitive (local) definition of an orbifold.

\begin{proposition}[{Thurston~\cite[Proposition~13.2.1]{thur}}]\label{propdis}
If $\Ma$ is a manifold and $\Gamma$ is a group acting properly 
discontinuously on $\Ma$, then $\Ma / \Gamma$ 
has an orbifold structure.
\end{proposition}

The above proposition shows that the structure of an
orbifold can be reasonably wild. 
Next, we present a definition that will be used 
to introduce the concept of a {\em good orbifold}.

\begin{definition}
Let $\Oc$ and $\wt{\Oc}$ be two orbifolds
with respective base spaces $X,\,\wt{X}$. We say
that $\wt{\Oc}$ is a
{\em covering orbifold} of $\Oc$ if there is a 
projection
$p\colon \wt{X} \longrightarrow X$ such that
each $x \in X$ admits a neighborhood 
$U = \tilde{U}/\Gamma$ ($\tilde{U}$ is an open subset of
$\R^{n}$) for which each component
$v_{i}$ of $p^{-1}(U)$ is isomorphic to 
$\tilde{U}/\Gamma_{i}$, where $\Gamma_{i}$ is a subgroup of
$\Gamma$ and the isomorphism respect the projections
\begin{center}
\includegraphics[scale=0.25]{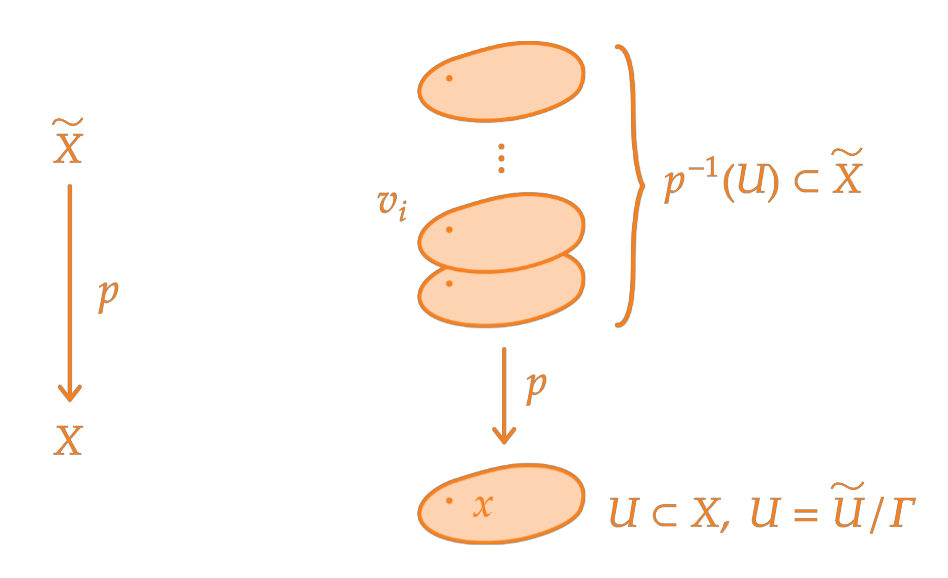}
\end{center}
\end{definition}

\begin{remark}
The main distinction between the concept of covering orbifolds
and the usual notion of covering spaces of topological spaces
is that when an orbifold $\wt{\Oc}$ covers another
orbifold $\Oc$, the components of the
inverse images of an open set $U$ of $\Oc$ may not be 
isomorphic to themselves or even to $U$. 
This is because each component $v_i$ that projects onto $U$
is a quotient of an open set of $\R^n$ by a subgroup 
$\Gamma_{i} \subset \Gamma$, see
the example below.
\end{remark}

\begin{example}
It is not difficult to see that the {\em pillowcase orbifold}
$\R^2/H$
of Example~\ref{pillowcaseeg} is a covering orbifold
of the rectangular billiard $\R^2/G$ 
of Example~\ref{billiardeg}. However, each component of
the inverse images of neighborhoods of each of 
the four vertexes $P_1,\,P_2,\,P_3,\,P_4$
of $\R^2/G$ are not isomorphic to
the image of the projection.
\begin{figure}
\centering
	\includegraphics[scale=0.55]{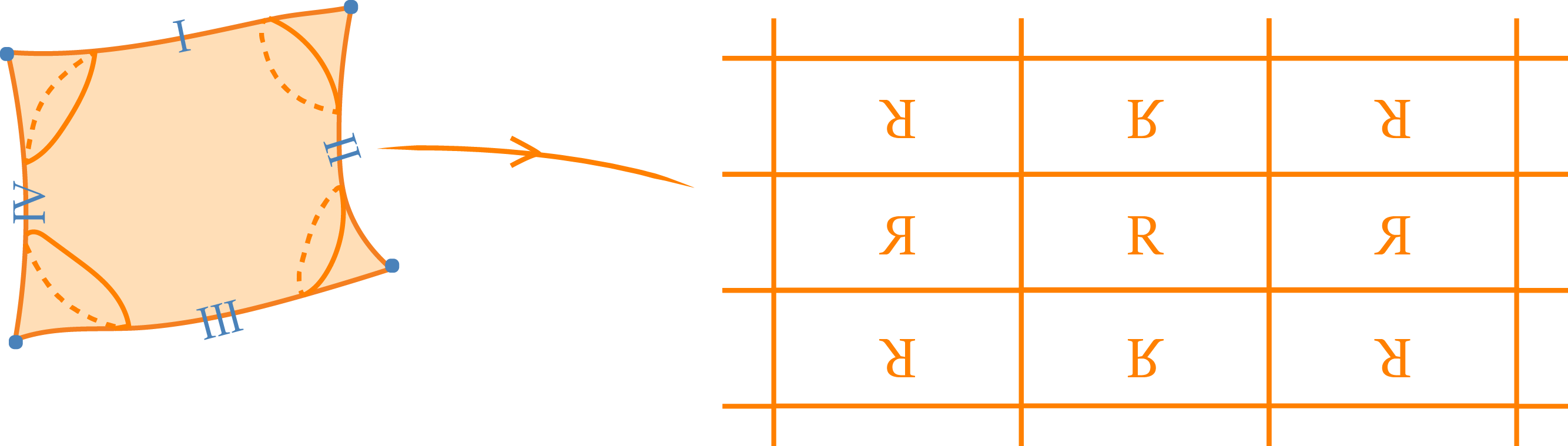}
\caption{A rectangular billiard with the pillowcase orbifold
as a covering orbifold.}
\end{figure}
\end{example}

\begin{definition}
An orbifold $\Oc$ is called a {\em good} orbifold if $\Oc$
admits a covering orbifold that is a differentiable manifold.
Otherwise, $\Oc$ is called a {\em bad} orbifold.
\end{definition}

Since the pillowcase and the billiard orbifolds are both covered by 
$\R^2$, they are both good orbifolds. In fact, in dimension two 
there are only a few bad orbifolds and their classification
can be seen in~\cite[Theorem~13.3.6]{thur}.

\section{3-manifolds}\label{sec3manf}

After noticing that geometry and topology may be deeply related
and presenting some initial concepts such as orbifolds,
we will now start the study of the geometric topology
of 3-manifolds, with the fundamental goal
of presenting the {\em geometrization theorem} (for orientable 3-manifolds).

Recall Theorem~\ref{geomSf}, where given an orientable, closed
surface $S$, we could find a {\em special metric} for $S$,
as being a metric of constant curvature $k$ being $-1,0$ or $1$,
the number $k$ depending uniquely on the topology of $S$.
The attempt to generalize this result to the class
of orientable closed 3-manifolds is quite natural.
However, (presently) it is easy to see that the
manifold $\sn2\times\sn1$ is an orientable, closed 3-manifold
that does not admit a metric of constant curvature. Indeed if
that was the case, its Riemannian universal cover would be
$\sn2\times\R$ with a metric of constant curvature. Since
$\sn2\times\R$ is simply connected, it would be isometric 
to a space form and diffeomorphic to either $\sn3$ or to $\R^3$.

This creates a great difficulty on understanding what could be
{\em the best metric} for a 3-manifold. And even Poincaré,
after proving the Uniformization Theorem for surfaces,
struggled with this question. After trying to generalize this
to 3-manifolds, he noticed it wouldn't be an easy task.
To get closer to this geometric classification,
Poincaré developed several topological concepts, such as 
the homology groups and the fundamental group of a manifold, 
arriving at the following question~\cite[Page 110]{PoincConj}.

\begin{center}
{\em
Consider a compact 3-dimensional manifold V without boundary. 
Is it possible that the fundamental group of V could be trivial, 
even though V is not homeomorphic to the 3-dimensional sphere?}
\end{center}

\noindent After adding that 
{\em cette question nous entraînerait trop loin} 
(this question would take us too far) and trying to prove 
this result, his question became known as the {\em Poincaré's 
Conjecture}, a widely known problem that, in the year 2000,
was deemed by the Clay Mathematics Institute, one
of the seven {\em millennium problems} (presently, the only one
with a complete solution).

Over the next sections, we will present
the theory of the topology and geometry of 3-manifolds, showing 
the notion of a {\em model geometry} and explaining 
the geometrization theorem, which proves (and generalizes)
Poincaré's conjecture by decomposing any orientable, closed
3-manifold into components, each of which admits
a model geometry that depends uniquely on its topology.

\subsection{Seifert fibered spaces}
The next step towards geometrization is to present the concept
of a {\em Seifert fibered space}. 
There are several ways of introducing those spaces, and we
chose to present them from a geometric point of view,
as 3-manifolds that
admit a decomposition by circles (or fibers) with a certain
structure called a {\em Seifert fibration}. In particular,
we must distinguish these two nomenclatures: 
a Seifert fibered space
will be the total space of a Seifert fibration (and a Seifert fibered
space may have distinct Seifert fibrations related to it).

As suggested by the denomination of the manifolds studied in this section, 
Seifert fibered spaces were first studied
by H. Seifert in the 1930's, with the intention of getting 
closer to the topological classification of closed 3-manifolds (or the 
``{\em homeomorphism problem for 3-dimensional closed manifolds}''). 
In~\cite{Seifertpaper} (see the 
book~\cite{Seifertbook}, which contains a geometric introduction to
Topology and also an English 
translation of~\cite{Seifertpaper}),
Seifert was able to completely classify, up to fiber-preserving 
homeomorphisms, his fibered spaces ({\em gefaserte Räume}),
and his research was a fundamental step towards geometrization, 
as it will become clear in the JSJ
decomposition presented in Theorem~\ref{jsj} below.

Before we present the precise definition of a Seifert fibered space, 
we will introduce the concept of a {\em fibered solid torus}.
Consider the (closed) unit
disk of dimension two, 
$\D^{2}: = \{ (x, y) \in \R^{2} \ : \ x^{2}+y^{2} \leq 1 \}$
and let the {\em trivial fibered solid torus}
be the product $\D^2\times \sn1$, endowed with the product 
foliation by circles, so that, for each $y\in \bbD^2$, 
$\{y\}\times\sn1$ is a fiber.
\begin{figure}[h] 	
\centering
	\includegraphics[scale=0.20]{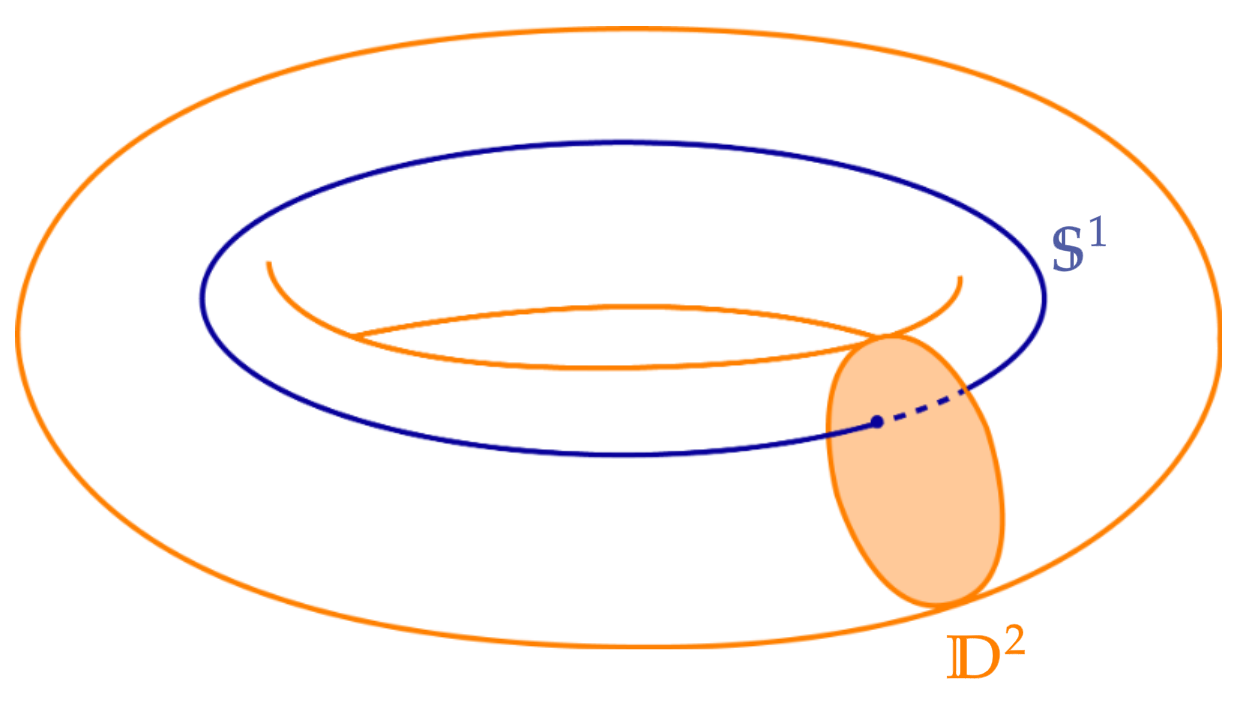}
\caption{Trivial fibered solid torus: each fiber is
of the type $\{y\}\times \sn1$, for $y\in \bbD^2$.}
\end{figure}

This trivial decomposition of $\bbD^2\times\sn1$ by circles
is a particular case of the next more general construction, which,
in the context of orientable Seifert
fibered spaces, provides the local picture.

\begin{definition}
Given a pair $p,\,q$ of co-prime integers with $p> 0$, 
the {\em standard fibered solid torus} $T(p,q)$ is 
obtained from the
trivial fibered solid torus $\D^2\times \sn1$ 
by cutting along
a disk $\D^2\times \{y\}$ and gluing it back together 
after a twist of $\frac{2q\pi}{p}$ in one of its sides,
i.e.,
$$ T(p,q) = 
\frac{\D^{2} \times [0,1] }{ \{ (z,0)  \sim  ( \psi _{p,q} (z), 1) \} },$$
where $ \psi _{p,q} : \D^{2} \rightarrow \D^{2} $ is defined
by $\psi _{p,q} (z) = e^{\frac{2\pi q}{p}i} z$. Then, $T(p,q)$ is a
solid torus, naturally endowed with a fibration by circles where
we have a {\em central fiber} (or {\em core fiber}) that comes
from $\{0\} \times [0,1]$ and any other fiber 
rotates $p$ times around the generator
of the fundamental group of the solid torus and $q$ times
around the central fiber.
\end{definition}

\begin{example}
Let $p = 8,\,q = 3$ and let $T(8,3)$ be the corresponding
standard fibered solid torus, given by the identification
$(\bbD^2\times[0,1])/\sim$ where $\bbD^2\times\{0\}\ni (z,0) \sim 
(e^{\frac{3\pi}{4}i}z,1)\in \bbD^2\times\{1\}$.

Let $x\in \bbD^2$ be given. If $x = 0$, then the line
$\{0\}\times [0,1]$ in $\bbD^2\times[0,1]$ descends
to the central fiber of $T(8,3)$. Otherwise, let $x_1 = x$
and let $x_{j+1} = e^{\frac{3\pi}{4}i}x_j$, for $j\in \N$. Then
$x_9=x_1$ and the union of the eight lines $\{x_j\}\times[0,1]$, 
$j=1,2,\ldots,8$,
descend to $T(8,3)$ as one fiber, that intersects 
any of the meridianal disks $\bbD^2 \times\{t\}$ in $T(8,3)$ eight times.
\begin{figure}[H] 	
\centering
	\includegraphics[scale=0.46]{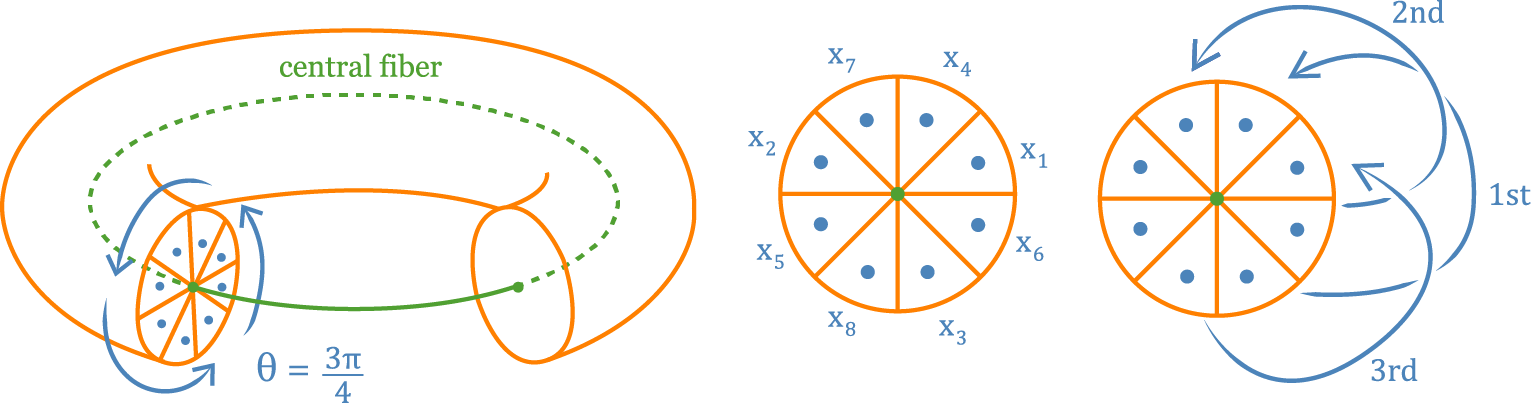}
\caption{$x_{j}$ is the intersection of the fiber with the 
disk after $j-1$ turns.}
\end{figure}
\end{example}

In the torus $T(p,q)$, $p$ is called the 
{\em multiplicity of the central fiber}.
When $p>1$, we say that the central fiber is {\em singular},
because this fiber goes around the solid torus
one time, while all the other fibers go around $p$ times.
Otherwise, the central fiber is called {\em regular}.

Having defined the structure of the standard fibered 
solid torus, we will next present the concept of a
Seifert Fibered space in the context of orientable 3-manifolds.

\begin{definition}\label{defSFSpace}
A {\em Seifert fibered space} is an orientable 3-manifold $\Ma$
that admits a decomposition into circles (called {\em fibers})
such that each fiber admits a neighborhood $U$ that is the union of
other fibers and $U$ is isomorphic (as
a fibered space) to a standard fibered solid torus.
A fiber of a Seifert fibered space is called {\em regular}
if it admits a neighborhood isomorphic (as a fibration) to the
trivial fibered solid torus. Otherwise, it is called
{\em singular}.
\end{definition}

We notice that the same manifold $\Ma$ may admit 
more than one distinct decomposition
by circles. Hence, whenever we say that $\Ma$ is a Seifert fibered
space, we are assuming that the circle decomposition is fixed. 
We also observe that there are manifolds which do not admit
any such decomposition, such as open simply 
connected spaces, see~\cite{Seifertpaper,Seifertbook}.

\begin{example}\label{ex.r3}
Let $C = [0,1]\times[0,1]\times[0,1]$ be a solid cube and
let $\cC$ be the 3-manifold obtained by identifying
the opposite sides of $C$ as follows (see Figure~\ref{figex.r3}):
$$(0,y,z) \sim (1,y,z)\text{ (left to right),}\quad
(x,0,z)\sim (x,1,z)
\text{ (front to back)},$$
$$(x,y,0)\sim (1-x,1-y,1)\text{ (top to a 180}^\circ\text{ rotation of the bottom)}.$$

\begin{figure}[H] 	
\centering
	\includegraphics[scale=0.5]{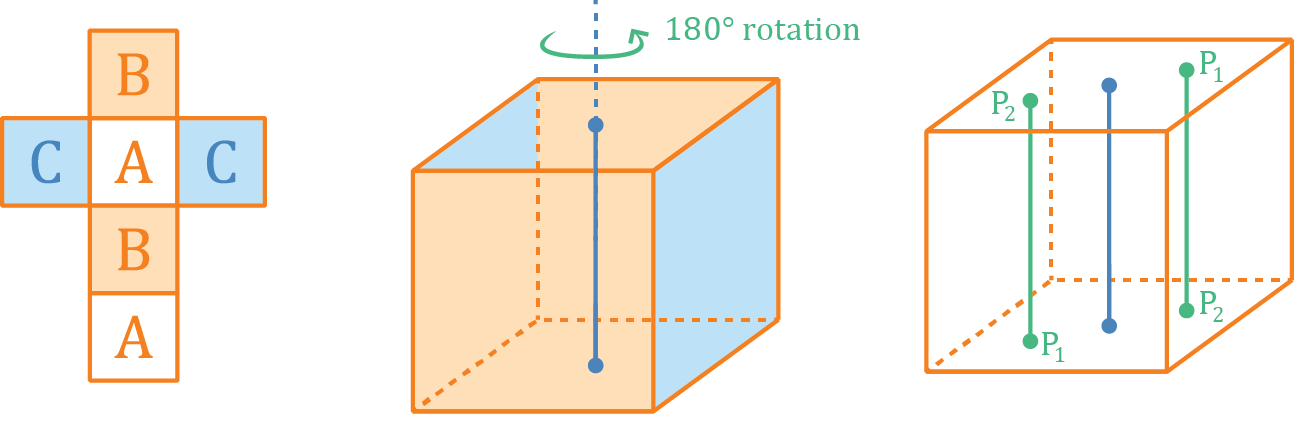}
\caption{\label{figex.r3}
Identifications of the cube constructing the Seifert fibered
space of Example~\ref{ex.r3}.}
\end{figure}
We may see that
the closed, orientable 3-manifold $\cC$ is a Seifert fibered 
space. Indeed, $\cC$ admits a decomposition in circles
that arise from the vertical lines in $C$, joining the top
side to the bottom side of $C$. In the remainder of this
example, we will let $I_{(x,y)}\subset \cC$ denote the equivalence
class of the vertical segment $\{(x,y,t)\mid t\in[0,1]\} \subset C$
and we will let $\gamma_{(x,y)}$ denote the respective fiber in
$\cC$ that contains $I_{(x,y)}$.
We have that, for all $(x,y)\in [0,1]\times[0,1]$, 
$\gamma_{(x,y)} = I_{(x,y)}\cup I_{(1-x,1-y)} = \gamma_{(1-x,1-y)}$ 
is a regular fiber
with the following four exceptions, each of which has
a neighborhood isomorphic
to $T(2,1)$ (see Figure~\ref{figcubefiber}).
\begin{itemize}
\item $\gamma_{(0,0)}=\gamma_{(1,0)} =\gamma_{(1,1)} = \gamma_{(0,1)}$,
which comes from $I_{(0,0)} = I_{(0,1)} = I_{(1,1)}= I_{(1,0)}$;
\item $\gamma_{(0,1/2)} = \g_{(1,1/2)}$, which comes from 
$I_{(0,1/2)} = I_{(1,1/2)}$;
\item $\gamma_{(1/2,0)} = \g_{(1/2,1)}$, which comes from 
$I_{(1/2,0)} = I_{(1/2,1)}$;
\item $\gamma_{(1/2,1/2)}$, which comes from $I_{(1/2,1/2)}$.
\end{itemize}

\begin{figure}[H] 	
\centering
	\includegraphics[scale=0.5]{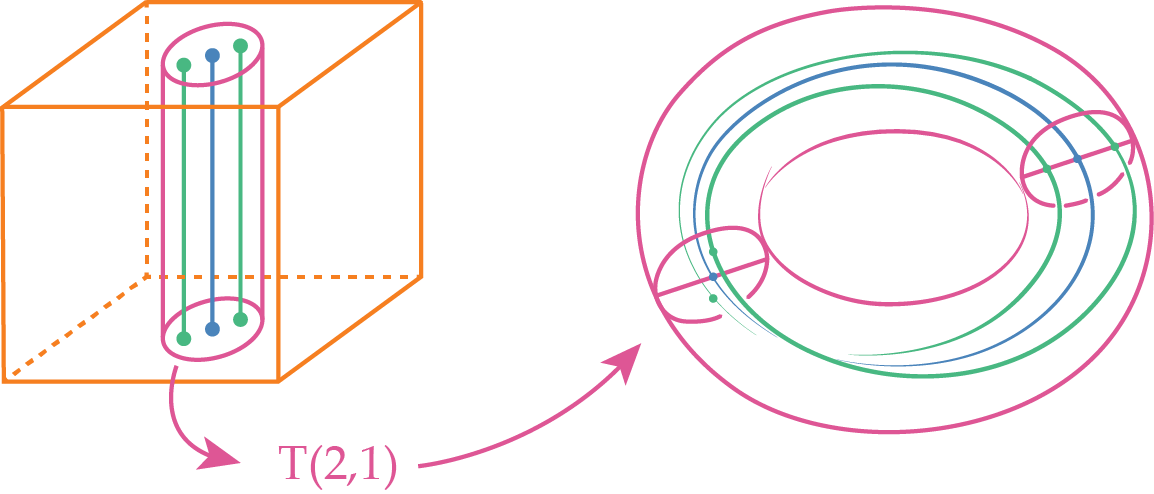}
\caption{Identification of a neighborhood of $\gamma_{(1/2,1/2)}$ in $\cC$
and $T(2,1)$.\label{figcubefiber} 
}
\end{figure}
\end{example}

A fundamental result by Epstein~\cite{epstein} 
is that in the context of 3-manifolds, a foliation by circles
is equivalent to being a Seifert fibered space.
\begin{theorem}[Epstein]
If $\Ma$ is a compact 3-manifold that admits a foliation by circles,
then $\Ma$ has the structure of a Seifert fibered space.
\end{theorem}

The topology of a Seifert fibered space can be understood
in terms of the way the fibers lie in the manifold. In 
this sense, we introduce the definition of a {\em Seifert fibration}
and of the {\em base space} of a Seifert fibered manifold.

\begin{definition}[Seifert fibration]\label{fibrs}
Let $\Ma$ be a 3-dimensional Seifert fibered space
and let $X$ be the topological space obtained
by collapsing each fiber of $\Ma$ to a point.
More precisely, $X$ is the quotient of $\Ma$ by
the equivalence relation $x\sim y$ if and only if $x$ and $y$
are in the same fiber of $\Ma$. 
We say that $X$ is the {\em base space} for $\Ma$.
Furthermore, if $\pi$ is the map 
$\pi\colon \Ma \to X$ that maps a point $x$ 
in a given fiber of 
$\Ma$ to the equivalence class $[x]\in X$,
we say that the triple
$(\Ma, X, \pi)$ is the {\em Seifert fibration} of $\Ma$.
\end{definition}

One of the most important properties of $X$ as above is that it
admits a structure of a good orbifold, where each singular fiber 
of $\Ma$ becomes a cone point in $X$ and $\pi\colon \Ma\to X$ is
a fibration in the orbifold sense.
With that in mind, we will say that $X$
is the {\em base orbifold} of $\Ma$.

\begin{proposition}\label{propbaseorbifold}
Let $\Ma$ be a 3-dimensional Seifert fibered space and let
$(\Ma,X,\pi)$ be the Seifert fibration of $\Ma$.
Then, $X$ admits the structure of a good orbifold.
\end{proposition}

A rigorous proof of Proposition~\ref{propbaseorbifold} can
be found in~\cite[Chapter~3]{scott}. We next present
a pictorial description of this good orbifold structure.

\begin{proof}[Idea of the proof of 
Proposition~\ref{propbaseorbifold}]
Let $\Ma$ be a Seifert fibered space with respective
Seifert fibration $(\Ma,X,\pi)$. Let $x$ be a point in $X$
with respective fiber $S = \pi^{-1}(\{x\})\subset \Ma$. Then, 
$S$ admits a neighborhood $V$, composed by fibers,
that descends to a a neighborhood $U$ of $x$ in $X$.
First, assume that $V$ is
isomorphic to a standard fibered solid torus
$T(p,q)$. Then, if $p = 1$, $U$ is diffeomorphic to the
disk $\bbD$ and if $p>1$, $U$ is isomorphic to a neighborhood
of the cone point in the orbifold $\R^2/\Z_p$. On the other hand, 
if $V$ is isomorphic to a standard fibered Klein bottle,
$X$ will be an orbifold with nonempty boundary and
$U$ is double-covered by $\bbD$ with a $\Z^2-$action, being
a neighborhood of a line of reflection.
\end{proof}

\begin{example}
Let $\cC$ be the Seifert fibered space given by Example~\ref{ex.r3}. Next,
we will show that the base orbifold $X$ of $\cC$ is the pillowcase orbifold
of Example~\ref{pillowcaseeg}. 
The fibers of $\cC$
are parameterized by $(x,y) \in [0,1]\times[0,1]$, so 
the quotient of $[0,1]\times[0,1]$ by the identifications
$(x,y)\sim (1-x,1-y)$, $(x,0)\sim (x,1)$ and $(0,y)\sim (1,y)$
give the structure of $X$. It is not difficult to see that
this structure is the same as $D = [0,1]\times[0,1/2]$, glued
along its boundaries by $(0,y)\sim (1,y)$, $(x,0)\sim(1-x,0)$ and
$(x,1/2)\sim(1-x,1/2)$ (see Figure~\ref{newfigure}), showing
that $X$ is topologically a sphere
with four cone points which is the pillowcase orbifold.
\end{example}

\begin{figure}
\centering
\includegraphics[width=0.5\textwidth]{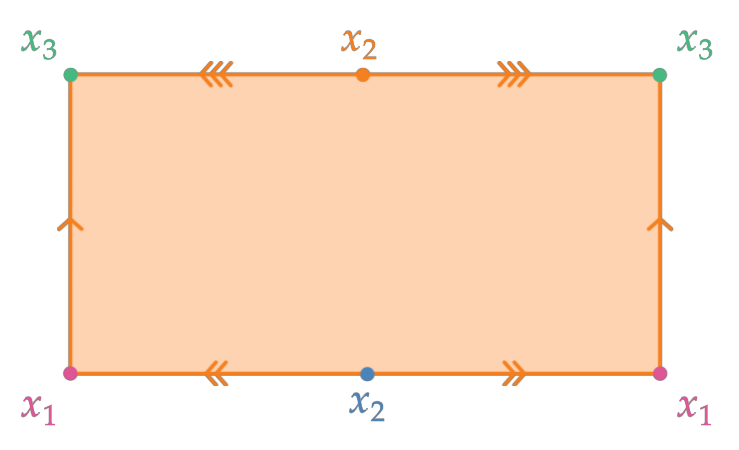}
\caption{The domain $D$ with the respective identifications
along its boundary. The points $x_1$ (representing $\gamma_{(0,0)}$),
$x_2$ (representing $\gamma_{(1/2,0)}$), $x_3$ (representing 
$\gamma_{(0,1/2)}$) and $x_4$ (representing $\g_{(1/2,1/2)}$)
are the four cone points of the orbifold $X$, each of which comes
from a singular fiber of $\cC$.\label{newfigure}}
\end{figure}

\subsubsection{Classification and the Euler number}

As seen previously, we may interpret a Seifert fibered space 
as a
fibration over an orbifold where the fibers are all diffeomorphic
to $\sn1$. Along this section, we will follow the construction of
A. Hatcher~\cite[Section~2.1]{Hatcher} to
present a classification (up to isomorphisms\footnote{We say
that two Seifert fibrations $(\Ma_1,X_1,\pi_1)$ and
$(\Ma_2,X_2,\pi_2)$ are isomorphic if there exists a 
diffeomorphism $\varphi\colon \Ma_1\to \Ma_2$ that
carries the fibers of $\pi_1$ to the fibers of $\pi_2$.})
of Seifert fibered spaces and introduce
a topological invariant called the Euler number. 
This invariant will be
useful both for distinction of Seifert fibered spaces and also
for identifying existence of horizontal surfaces (see 
Definition~\ref{defHor}) on closed Seifert fibered spaces.

We start this discussion by presenting a general
construction of a Seifert fibered space.

Let $S$ be a compact surface ($S$ is either orientable or 
non orientable, and has possibly nonempty boundary). 
For a given $n\in \mathds{N}$, let
$D_1,\,D_2,\,\ldots,\,D_n$ be a pairwise disjoint collection of
closed disks in the interior of $S$ and let $S'$ be the closure of 
$S\setminus (D_1\cup\ldots\cup D_n)$, so $S'$ is a 
compact surface with nonempty boundary.
Let $M'$ be the total space of the orientable circle bundle over
$S'$. Specifically, if $S'$ is orientable,
$M'=S'\times \sn1$. Otherwise, $M' = S'\wt{\times}\sn1$ 
can be defined as follows.
Let $S'$ be given by an identification of pairs $a_i,\,b_i$
of oriented arcs in the boundary of a 
topological closed disk $D$ 
(see the example in Figure~\ref{figKlein}). Then,
$M'$ can be obtained from $D\times \sn1$ by identifying the
surfaces $a_i\times\sn1$ and $b_i\times \sn1$ via the product of the
given identification $a_i\sim b_i$ with either the identity or 
a reflection in the $\sn1$ factor, whichever makes $M'$ orientable.
In particular, $M'$ is compact and each boundary component of $M'$ has
the topology of a 2-torus.

The construction presented above allows us to see the circle bundle
$\pi\colon M'\to S'$ as the double of an $[0,1]$-bundle,
thus, there exists a well defined global cross section
$\sigma\colon S'\to M'$, i.e.,
$\sigma$ is a continuous function and $\pi(\sigma(x)) = x$
for all $x\in S'$.
Next, we will make use of $\sigma$, together
with an orientation on $M'$, to obtain 
a well-defined notion of {\em slopes}\footnote{Recall
that the 2-torus $\bbT^2 = \sn1\times\sn1$ has 
universal covering map defined by 
$(x,y)\in \R^2\mapsto (e^{2\pi x i},e^{2\pi y i})$, and
that a line $\{y = \alpha x\}\subset \R^2$ descends to $\bbT^2$ as
a simple closed curve $c_\alpha \subset \bbT^2$ 
if and only if $\alpha \in \mathds{Q}$. We may
also extrapolate this definition to allow the curve 
$\{x = 0\}$ to be seen as the $\alpha = \infty$ case. Moreover,
using this model, a nontrivial
simple closed curve in $\bbT^2$ is always
isotopic to a unique curve $c_\alpha$ as defined above, for
some $\alpha \in \mathds{Q}\cup \{\infty\}$, and the number
$\alpha$ is defined as the {\em slope} of the curve.}
for nontrivial simple closed curves in the
boundary components of $M'$.
Fix $T$ a component of $\partial M'$ and let 
$m= \sigma(\sigma^{-1}(T))$. Since $\sigma^{-1}(T)$ is a boundary
component of $S'$, $m$ is
a nontrivial closed curve in $T$. Now, choose any $p\in m$ and
consider the fiber over $p$, $l = \pi^{-1}(\{p\})$. Once again,
$l$ is a nontrivial closed curve in $T$ and, since $\sigma$
is a cross section of $\pi$, $m\cap l = \{p\}$. Moreover,
the curves $m$ and $l$ (which are also known as {\em natural
curves} for $T$, see Figure~\ref{fignat}) generate $\pi_1(T,p)$. Then, there 
exists a diffeomorphism $\varphi\colon T \to \sn1\times\sn1$ such that
$\varphi$ maps $m$ to $\sn1\times\{0\}$ (a slope 0 curve)
and $l$ to $\{0\} \times \sn1$ (a slope $\infty$ curve). For
simplicity, when we talk about the slope
$\frac{r}{s}$ of a given curve, we are assuming that 
$r$ and $s$ are co-primes.

\begin{figure}
\centering
\includegraphics[width=0.6\textwidth]{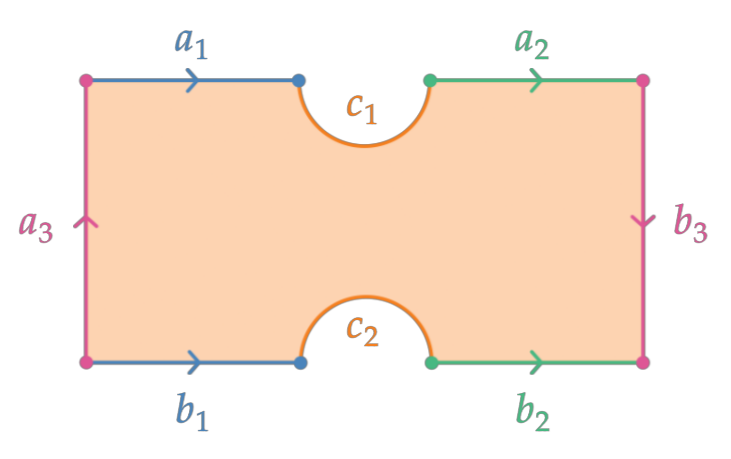}
\caption{\label{figKlein}
Let $S'$ be the Klein bottle with
one disk removed. Then, $S'$ can be constructed from a topological 
disk $D$ (highlighted in the above figure) by identifying the
oriented arcs $a_1\sim b_1$, $a_2\sim b_2$ and $a_3\sim b_3$. Note 
that $c_1$ and $c_2$ are not identified other than by
its shared endpoints with the arcs $a_1,b_1,a_2$ and $b_2$.}
\end{figure}

The next step in our construction of a Seifert fibered 3-manifold
is to {\em fill} the boundary components of $M'$ that 
are generated from $\partial D_1,\,\ldots,\,\partial D_n$ by
attaching solid tori to them.
By an abuse of notation, in the 
remainder of this construction when we choose {\em a boundary 
component} of $M'$ we will assume, without further comments, 
that it is one of the $n$ components 
we just described. Note that
for each such component, there are infinitely many
ways of doing such a gluing, which is called a
{\em Dehn filling}. Let $T$ be a boundary component of $M'$
with natural curves $m$ and $l$ as defined in the previous discussion.
Then, after choosing orientations, 
$(m,l)$ provides a positively oriented basis for the 
first homology group $H_{1}(T,\Z)$ and the curve with slope 
$\frac{r}{s}$, defined as
$\gamma = r l+s m$ 
is prime, in the sense that it is not
a multiple $k \gamma'$ (in homology) of another curve $\gamma'$ unless 
$\vert{k}\vert = 1$ and $\gamma'=\pm \gamma$.

\begin{definition}
The Dehn filling of $T$
generated by $\frac{r}{s}\in\Q$
is the unique (up to homeomorphism)
manifold generated by gluing a solid torus 
$\D\times\sn1$ to $T$ by its boundary in such a way that the 
boundary of the {\em meridianal disk} $\D\times\{0\}$ 
is glued (by a diffeomorphism) to a curve
of slope $\frac{r}{s}$ in $T$. In other 
words, the Dehn filling of $T$ determined 
by $\frac{r}{s}$ glues a solid torus to $T$, making the 
curve $rl+sm$ (and therefore
any of its multiples) trivial in homology.
\end{definition}

\begin{figure} 	
\centering
	\includegraphics[scale=0.22]{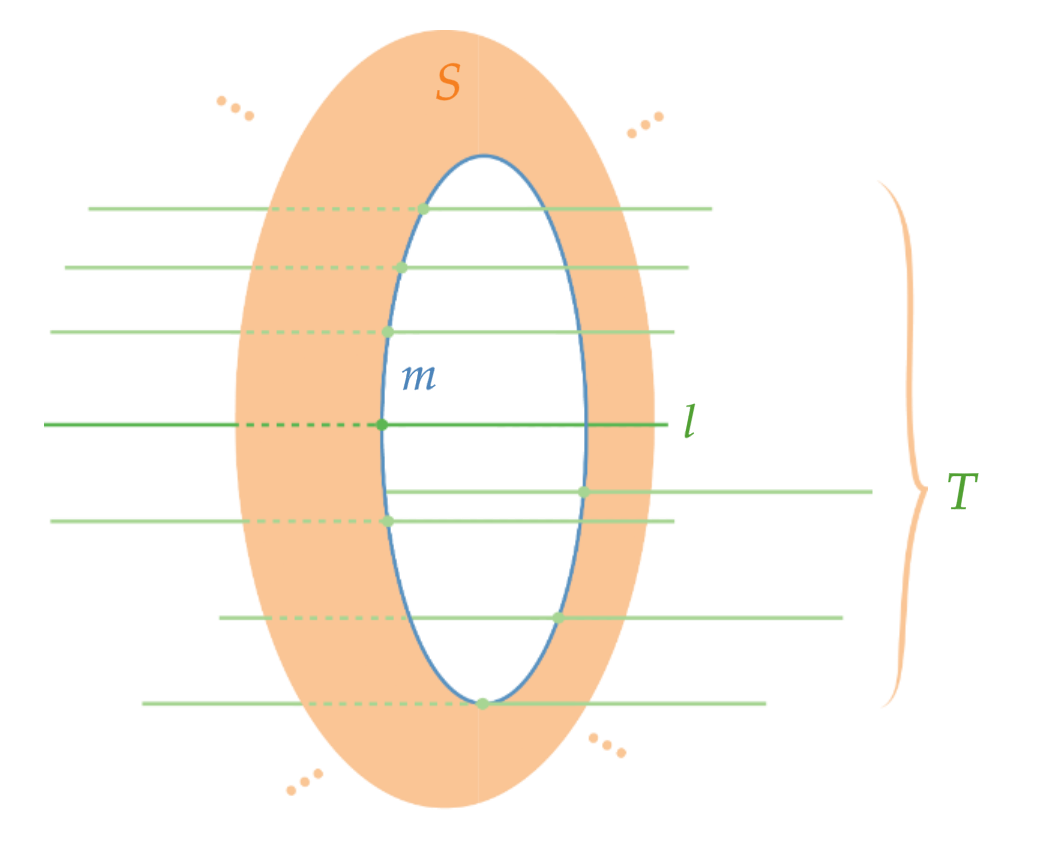}
\caption{\label{fignat}
The natural curves $m$ and $l$ on a torus boundary component $T$ of a 
Seifert fibered space.}
\end{figure}

Let $M^1$ be the manifold obtained by performing
a Dehn filling generated by a slope $\frac{r_1}{s_1}$ 
in a torus boundary of $M'$. Then,
the circle bundle over $M'$ extends naturally to 
a circle bundle over $M^1$, since the fibers (slope $\infty$) 
are not isotopic to meridian circles in the attached $\D\times\sn1$.
Intuitively, the Dehn-filling as above glues a neighborhood
with the structure of a $T(s_1,r_1)$ fibered solid torus to
the original Seifert fibration of $M'$. In particular,
the base space of the new fibration has the
structure of an orbifold (possibly with boundary)
which has one cone point of multiplicity $s_1$ (if $s_1 = 1$, 
the fiber is regular).

The above observation that the original circle
bundle $M'\to S'$ extended after performing one Dehn filling in 
one boundary component of $M'$ allows us to repeat the process,
generating the following resulting manifold.

\begin{definition}\label{defSM}
Let $S$ be a compact surface and let 
$\frac{r_1}{s_1},\,\frac{r_2}{s_2},\,\ldots,\,\frac{r_n}{s_n}\in \Q$.
Then, the 3-manifold
$$M(S,\,\frac{r_1}{s_1},\,\frac{r_2}{s_2},\,\ldots,\,\frac{r_n}{s_n})$$
is the resulting Seifert-fibered space after performing 
$n$ Dehn fillings with slopes $\frac{r_i}{s_i}$ on the boundary 
components of $M'$ as previously described.
\end{definition}

Note that, by construction,
$M(S,\,\frac{r_1}{s_1},\,\frac{r_2}{s_2},\,\ldots,\,\frac{r_n}{s_n})$ 
has a Seifert fibration over the orbifold 
$(S, x_{1}, ..., x_{n})$, 
where each $x_i$ is a cone point in $S$ with multiplicity $s_i$.
There are several questions regarding this construction, and 
we expect the next proposition to answer many, if not all of them.
A proof of it can be found in Hatcher~\cite[Proposition~2.1]{Hatcher}.

\begin{proposition}\label{propClass}
Using the notation introduced by Definition~\ref{defSM}, the following hold:
\begin{enumerate}
\item Every compact, orientable Seifert 
fibered 3-manifold is isomorphic
to one of the models
$M(S,\,\frac{r_1}{s_1},\,\frac{r_2}{s_2},\,\ldots,\,\frac{r_n}{s_n})$.
\item 
$M(S,\,\frac{r_1}{s_1},\,\frac{r_2}{s_2},\,\ldots,\,\frac{r_n}{s_n})$
is isomorphic to
$M(S,\,\frac{r_1}{s_1},\,\frac{r_2}{s_2},\,\ldots,\,\frac{r_n}{s_n},0)$.
\item 
$M(S,\,\frac{r_1}{s_1},\,\frac{r_2}{s_2},\,\ldots,\,\frac{r_n}{s_n})$
and 
$M(S,\,-\frac{r_1}{s_1},\,-\frac{r_2}{s_2},\,\ldots,\,-\frac{r_n}{s_n})$
are related by a change of orientation.
\item \label{it4}
$M(S,\,\frac{r_1}{s_1},\,\frac{r_2}{s_2},\,\ldots,\,\frac{r_n}{s_n})$
is isomorphic to
$M(S,\,\frac{r_1'}{s_1'},\,\frac{r_2'}{s_2'},\,\ldots,
\,\frac{r_n'}{s_n'})$ by an orientation preserving diffeomorphism
if and only if the following two conditions hold:
\begin{enumerate}
\item After a permutation of indices,
it holds, for
all $i\in \{1,2,\ldots,n\}$,
$\frac{r_i}{s_i}\equiv \frac{r_i'}{s_i'} {\rm mod}\ 1$.
\item \label{it4b} If $\partial S = \emptyset$, 
$\displaystyle\sum_{i=1}^n \frac{r_i}{s_i} =
 \displaystyle\sum_{i=1}^n \frac{r_i'}{s_i'}$.
\end{enumerate}
\end{enumerate}
\end{proposition}

Proposition~\ref{propClass},
together with the construction of Definition~\ref{defSM},
gives a complete classification of Seifert fibrations, up to
isomorphisms. We note that, in order to obtain a classification
of Seifert fibered spaces up to diffeomorphisms, more work
still needs to be done, since there are diffeomorphic Seifert fibered
spaces which are not isomorphic as Seifert fibrations (in other
words, there are manifolds with more than one Seifert 
fibration structure). 
For such a classification, we
suggest~\cite[Chapter~10]{martelli} or~\cite[Theorem~2.3]{Hatcher}.

We are now ready to define an invariant of
a Seifert fibration, which is the {\em Euler number}. Note that
this invariant is well defined by item~\ref{it4b} of 
Proposition~\ref{propClass}.

\begin{definition}[Euler number]
Let $\eta$ be a Seifert fibration given by
$
M(S,\,\frac{r_1}{s_1},\,\frac{r_2}{s_2},\,\ldots,\,\frac{r_n}{s_n})$
with a closed, orientable total Seifert fibered space $\Ma$.
Then, the {\em Euler number} of $\eta$ is
$$e(\eta) = \sum_{i=1}^n\frac{r_i}{s_i}.$$
\end{definition}

Intuitively, other than being helpful for distinguishing
Seifert fibrations, the Euler number measures how far we are
from obtaining a section for the respective fiber bundle.
The next definition 
(and the following Proposition~\ref{prophorizonta})
make this intuition precise.

\begin{definition}\label{defHor}
Let $\Ma$ be a compact Seifert fibered space (together with
a Seifert fibration) and let
$\Sigma \subset \Ma $ be a closed surface embedded in $\Ma$. 
We say that $\S$ is {\em vertical} if it is a union of 
regular fibers (in this case, $\S$ is either a torus or a Klein
bottle whose projection over the base orbifold $X$ is a simple closed 
curve in the complement of the cone points of $X$).
On the other hand, we say that $\S$ is {\em horizontal}
if $\S$ is everywhere transverse to the fibers.
\end{definition}

A proof to Proposition~\ref{prophorizonta} can be
found in~\cite[Proposition~2.2]{Hatcher}.

\begin{proposition}\label{prophorizonta}
Let $\Ma$ be a compact, 
orientable Seifert fibered space with respective
Seifert fibration $\eta$. Then:
\begin{enumerate}
\item If $\partial \Ma \neq \emptyset$, then there exists
a horizontal surface in $\Ma$.
\item If $\partial \Ma = \emptyset$, then
there exists a horizontal surface in $\Ma$ if and only if
$e(\eta) = 0$.
\end{enumerate}
\end{proposition}

\subsection{Decomposition of 3-manifolds}

Having defined the concept of a Seifert fibered space, we are able
to present some of the main developments of the theory of 3-manifolds
before geometrization. The natural path for understanding any 
mathematical object is to try to break it up into simpler pieces,
and that was firstly obtained using the concept of 
connected sum (see Definition~\ref{consum}). Note that 
the 3-sphere $\sn3$ acts as the {\em neutral element} for 
the connected sum of 3-manifolds, since for any
3-manifold $M$, the connected sum $M\# \sn3$ is
diffeomorphic to $M$.

\begin{definition}
A 3-manifold $\Ma$ is called {\em prime} if any connected sum
$\Ma = \Ma_{1} \# \Ma_{2}$ is trivial in the sense that either
$ \Ma_{1}$ or $\Ma_{2}$ is the 3-sphere $\sn3$.
\end{definition}

Note that if a 3-manifold $\Ma$ is not prime, then there exists
a decomposition of $\Ma$ in a nontrivial
connected sum $\Ma = \Ma_1\# \Ma_2$. In particular, there is an 
embedded topological 2-sphere $S\subset \Ma$ that separates 
$\Ma$ into two regions, one diffeomorphic to $\Ma_1\setminus B^3$
and another diffeomorphic to $\Ma_2\setminus B^3$, where $B^3$
represents the 3-ball. Thus, we may introduce the 
closely related notion of an irreducible manifold as follows:

\begin{definition}\label{defirred}
We say that a 3-manifold $\Ma$ is {\em irreducible} if
any embedded 2-sphere in $\Ma$ is the boundary of a 3-ball in $\Ma$.
\end{definition}

It is straightforward to see that if $\Ma$ is irreducible,
then $\Ma$ is prime. The converse does not hold, since
for a given $p\in \sn1$, the sphere $\sn2\times \{p\}$ does not
bound any 3-ball in the prime 3-manifold
$\sn2\times \sn1$. But in fact,
the only closed, orientable prime 3-manifold that is
not irreducible is $\sn2\times\sn1$.

The next result, due to Kneser~\cite{kneser} and 
Milnor~\cite{milnor}, establishes that any closed orientable
3-manifold admits a {\em unique} decomposition
by prime factors.

\begin{theorem}[Kneser-Milnor]\label{thmKnMi}
Let $\Ma$ be a closed, oriented 3-manifold. Then, there are
closed, oriented, prime 3-manifolds $\Ma_1,\,\Ma_2,\,\ldots,\,\Ma_k$
such that $\Ma$ is homeomorphic to the connected sum
$\Ma_1\#\Ma_2\#\ldots\#\Ma_k$.
Furthermore, the nontrivial factors in this decomposition are
unique up to reordering and orientation-preserving diffeomorphisms.
\end{theorem}

Although Theorem~\ref{thmKnMi} cuts a closed, oriented 3-manifold
along spheres, providing a standard decomposition
over prime factors
(and thus along irreducible and $\sn2\times\sn1$ factors),
it is not sufficient to understand the topology
of 3-manifolds, since even restricting to
irreducible, closed, oriented 3-manifold, this classification is
not an easy task.
Another step towards this goal was to obtain another
standard decomposition, by the (in some sense) second
simplest topology of a surface, which is the decomposition
of any irreducible, closed and orientable 3-manifold along tori.
This decomposition (which will be presented in Theorem~\ref{jsj}
below) is called the JSJ decomposition, an acronym to the names
of the researchers that proved its existence: Jaco-Shalen~\cite{JS}
and Johannson~\cite{otherJ}. 
Before presenting the statement of this decomposition,
we need a few extra definitions.

\begin{definition}\label{incompr}
Let $\Ma$ be a compact 3-manifold and let $S$ be a surface 
properly embedded\footnote{In this setting, we say that
$S$ is proper if $S$ is 
compact and $\partial S= S\cap \partial M$.}
in $\Ma$. A {\em compression disk} $D$ for $S$ is an embedded disk
in $\Ma$ which intersects $S$ transversely,
and such that $\partial D = D\cap S$
does not bound a disk in $S$. Furthermore,
if $S$ admits a compression disk, we say that $S$ is
{\em compressible}, and if $S$ is not compressible and not a
2-sphere, we say that $S$ is {\em incompressible}.
\end{definition}

In some sense, the next definition 
gives the equivalent definition of a prime 3-manifold in the context
of a torus decomposition.

\begin{definition}[Atoroidal manifold]
Let $\Ma$ be a compact 3-manifold with empty or
toroidal boundary. 
We say that $\Ma$ is
{\em atoroidal} (or homotopically atoroidal) 
if any (immersed) incompressible torus is homotopic to 
a component of $\partial \Ma$.
\end{definition}

Closely related to the notion of atoroidal manifolds is the following
definition.

\begin{definition}[Geometrically atoroidal manifold]\label{defgam}
Let $\Ma$ be a compact 3-manifold with (possibly empty)
toroidal boundary. 
We say that $\Ma$ is {\em geometrically atoroidal}
if any embedded, incompressible torus is isotopic to 
a component of $\partial \Ma$.
\end{definition}

\begin{remark}\label{smallseifert}
Any
atoroidal manifold is 
geometrically atoroidal. However, the converse does not hold,
as it is easy to see by the example of Figure~\ref{figgeomat}.
Nonetheless, it is true that if $\Ma$ is a compact 3-manifold
with (possibly empty) toroidal boundary and $\Ma$ is
geometrically atoroidal,
then $\Ma$ is atoroidal unless it is 
a {\em small Seifert fibered space}, in the sense that
it is a Seifert fibered space and the base orbifold
has genus zero and the number of cone points, together with the 
number of boundary components, is at most three.
\end{remark}

\begin{figure}
\centering
\includegraphics[width=0.5\textwidth]{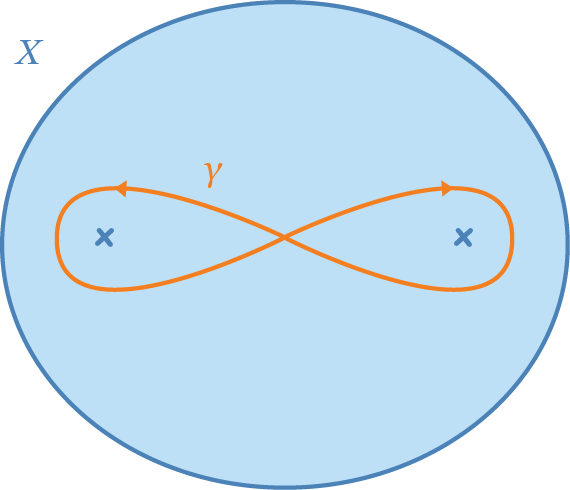}
\caption{\label{figgeomat} Let $X$ be a {\em pair of pants}, i.e. 
$X = \bbD\setminus\{p_1,\,p_2\}$ is the open disk, punctured twice.
Then, the manifold $\Ma = X\times\sn1$ is geometrically atoroidal
(since any torus in $\bbD\times\sn1$ separates) but not
homotopically atoroidal, since the torus $\gamma\times\sn1$, where
$\gamma$ is the curve depicted above, is incompressible
and not boundary-parallel. However,
$\Ma$ is a small Seifert fibered space.}
\end{figure}

Having defined the concepts of a Seifert fibered space
and of an atoroidal manifold, we may now present
the JSJ decomposition theorem.

\begin{theorem}[JSJ decomposition]\label{jsj}
Let $\Ma$ be an irreducible, compact and orientable 3-manifold with
empty or toroidal boundary. Then, $\Ma$ admits a (possibly empty) 
pairwise disjoint
collection $T_1,\,T_2,\,\ldots,\,T_n$ of embedded, incompressible tori 
that separate $\Ma$ into components, each of which is either
atoroidal or Seifert fibered. Furthermore, any collection of such tori
with the minimal number of elements is unique, up to isotopy.
\end{theorem}

Theorem~\ref{jsj} gives a good reason for we to work
with incompressible tori, since using this concept avoids 
artificial decompositions such as cutting $\sn3$ by some
{\em knotted torus}, separating it into a 
component that is a solid torus
and another component a knot complement 
(we will revisit knot complements in Section~\ref{sechyperb}).

After the JSJ decomposition, the next step in order to classify 
the topology of 3-manifolds was to understand atoroidal and
Seifert fibered compact 3-manifolds with (possibly empty) toroidal
boundary that appeared in a finer version of the JSJ decomposition
(see Remark~\ref{remgeodec} in Section~\ref{secgeomm}). 
However, this was
not an easy task, as it is easy to assume because the Poincaré 
Conjecture, which arguably 
dealt with the simplest possible topology among
3-manifolds, was still open
and would play a definite role on this subject.
The next step towards this goal was given by Thurston,
which put the Poincaré Conjecture as a particular case
in a broad context, by understanding that
that geometrization could be achieved for toroidal decompositions
for 3-manifolds. We remark that there are several equivalent (but 
perhaps not so easily seen as being equivalent) statements of the
geometrization in the literature, 
and we will focus in presenting the ones
that appear more geometric.

\subsection{Geometrization of 3-manifolds}\label{secgeomm}

The history of geometrization actually starts with Henri
Poincaré, which, after proving the Uniformization theorem
(Theorem~\ref{geomSf}, also known as the geometrization for surfaces),
developing several important tools for
geometry and topology, such as homology, homeomorphism and
fundamental group, started to wonder whether homology 
could be sufficient to characterize a topology 
(the answer is no, as Poincaré himself
answered with his example of a {\em homology sphere}), generating,
in this process,
the Poincaré Conjecture. 
But it was William Thurston that was able to tackle the 
problem of bringing the geometrization
for 3-manifolds. In fact, Thurston proved 
several
results concerning geometrization of 3-manifolds, including the fact
that there are 8 maximal 
{\em model geometries} (see Definition~\ref{ModelG} below)
of dimension 3 that admit 
compact quotients, and the geometrization itself for a broad class
of 3-manifolds. For his groundbreaking work, he was awarded with
a Fields medal in the ICM held in 
Warsaw, Poland, 1982\footnote{Actually, 
due to the introduction
of the Martial Law in Poland in December 1981, the
conference was postponed until 1983.}.

\begin{definition}[Model geometry]\label{ModelG}
A {\em geometry} $\bbX$ is a pair $(X,G)$ where $X$ is a
simply connected manifold and $G$ is a Lie group acting on 
$X$ transitively, via diffeomorphisms, and with 
compact isotropy\footnote{If $G$ is a group
acting on a manifold $X$, for each $x\in X$ the isotropy group
over $x$ is the stabilizer subgroup $G_x = \{g\in G\mid gx = x\}$.}
groups. When the group $G$ is maximal with respect to 
all Lie groups acting on $X$ transitively, via diffeomorphisms and
with compact isotropy groups, $\bbX = (X,G)$ is called
a {\em model geometry}\footnote{Note that if $(X,G)$ is a 
geometry that is not maximal, it
can always be extended to a model geometry $(X,G')$.}.
We say that two geometries $(X,G)$ and $(X',G')$ are
equivalent if there exists a diffeomorphism $X\to X'$ that maps
the action of $G$ to the action of $G'$.
\end{definition}

Geometrically, a model geometry can be seen as a simply connected
homogeneous manifold $X$, together with the action of its
full isometry group ${\rm ISO}(X)$, 
and two geometries are equivalent if
the corresponding Riemannian manifolds are isometric. 
Using this point of view, we say that
a manifold $\Ma$ has a {\em geometric structure} modelled over 
a geometry
$\bbX = (X,{\rm ISO}(X))$ if $\Ma$ admits a Riemannian metric such that
its Riemannian universal covering is isometric to $X$,
and we say that such a geometric structure is {\em complete} 
if this metric over $\Ma$ is complete. Next, we present a more
rigorous, group-theoretical, definition for a geometric structure.
 
\begin{definition}[Geometric structure]
Let $\Ma$ be a manifold. We say that $\Ma$ admits a
{\em geometric structure} over 
a geometry $\bbX = (X,G)$ if there exists
a diffeomorphism between $\Ma$ and $X/\Pi$, where
$\Pi$ is a discrete subgroup of $G$ acting freely on $X$.
\end{definition}

Using the concepts of a model geometry and of a geometric structure,
we notice that Theorem~\ref{geomSf} has the following interpretation.

\begin{theorem}[Geometrization of surfaces]
\label{geomSf2}
Let $S$ be an orientable, closed surface. Then, 
$S$ admits a geometric structure over a model geometry
$\bbX$. Furthermore,
if $g$ is the genus of $S$, it holds that
$g=0$ if and only if $\bbX = \sn2$, 
$g=1$ if and only if $\bbX = \R^2$ and 
$g\geq2$ if and only if $\bbX = \hn2$.
\end{theorem}

A consequence of Theorem~\ref{geomSf2} is that {\em any closed, 
orientable surface is geometrizable} and that there are 
3 possible model geometries (in fact, if a manifold admits a 
geometric structure over a model geometry $\bbX$, 
although the geometric structure is not unique, the geometry
is).
In the context of 3-manifolds, it is not difficult
to find examples of closed, orientable 3-manifolds that do not
admit any geometric structure; for instance, a nontrivial 
connected sum $M_1\# M_2$ does not admit a geometric structure,
with the unique exception of $\R\mathbb{P}^3\#\R\mathbb{P}^3$.

But even among orientable, closed irreducible
3-manifolds we may find examples that do not admit a geometric 
structure\footnote{Let $M_1$ and $M_2$ be compact, irreducible 
3-manifolds with
both $\partial M_1$ and $\partial M_2$ diffeomorphic to a 2-torus $T$.
In comparison to the connected sum, we may 
choose a gluing diffeomorphism 
$\varphi\colon \partial M_1\to \partial M_2$ to obtain 
a 3-manifold 
$M = (M_1\cup M_2)/(\partial M_1 \ni x~\varphi(x)\in \partial M_2)$.
Depending on $\varphi$ and on $M_1,\,M_2$, it holds that $M$ will be 
irreducible and not geometrizable.}.
Therefore, we need to consider a further 
decomposition to obtain geometrization.
The natural context to do so is to attempt 
to show that the pieces that appear in the JSJ decomposition
of an irreducible closed 3-manifolds (Theorem~\ref{jsj}),
which are compact manifolds with toroidal boundary, are geometric. 
More precisely, we have the following definition.

\begin{definition}\label{geom}
Let $\Ma$ be a compact 3-manifold with (possibly empty)
toroidal boundary. We say that $\Ma$ is {\em geometric} if
${\rm int}(\Ma)$ admits a complete geometric structure 
of finite volume.
\end{definition}

As explained above, the main goal of the geometrization is to 
obtain a standard decomposition of an orientable, irreducible, 
closed 3-manifold such that any component is geometric, and the
JSJ decomposition is the natural decomposition to start the 
analysis. However, there exists one final obstacle to 
avert.

\begin{remark}[Geometric decomposition]\label{remgeodec}
There is one {\em special} compact 3-manifold with toroidal boundary
that may appear as a component of the JSJ decomposition, which 
is the total space of the 
oriented twisted $[0,1]$-bundle over the Klein
bottle\footnote{Let $f\colon \bbT^2\to \bbK$ denote the oriented
double cover of the Klein bottle $\bbK$ by the 2-torus $\bbT^2$ and
let $\sigma\colon \bbT^2\to \bbT^2$ be the nontrivial covering
transformation related to $f$. Then, $\bbK\wt{\times}[0,1]$ is
diffeomorphic to the quotient $(\bbT^2\times[-1,1])/\varphi$, 
where $\varphi\colon \bbT^2\times[-1,1]\to \bbT^2\times[-1,1]$ is
defined as $\varphi(x,t) = (\sigma(x),-t)$.
Then, $\bbK\wt{\times}[0,1]$ is orientable and has a central
Klein bottle identified with $\bbT^2\times\{0\}/\varphi$.},
denoted by 
$\mathbb{K}\wt{\times}[0,1]$. This manifold
is not geometric, although its interior admits a complete
geometric flat (i.e., modelled over $\R^3$) structure, but of
infinite volume. In order to avoid the issue of a nongeometric
component appearing in a {\em good} decomposition by tori, if
$\Ma\neq \bbK^2\wt{\times}[0,1]$ is as in Theorem~\ref{jsj} and
$U$ is one of the components of a minimal JSJ decomposition
which is diffeomorphic to $\mathbb{K}\wt{\times}[0,1]$,
we may replace
the torus boundary $T = \partial U$ in the decomposition
by the central Klein bottle of $U$, thus generating a different
decomposition for $\Ma$. 
After performing this procedure on each component
diffeomorphic to $\mathbb{K}\wt{\times}[0,1]$, we obtain a 
new decomposition for $\Ma$, where each component is now 
a compact manifold with (possibly empty) boundary composed
of tori and Klein bottles. This decomposition (which, when minimal, is
also unique up to isotopy) is called the {\em geometric decomposition}
of $\Ma$, for more details see~\cite[Section~11.5.3]{martelli}.
\end{remark}

Having defined what we mean by a {\em geometric} manifold
and by the {\em geometric decomposition} of a closed, orientable
3-manifold, we will
next present Thurston's geometrization conjecture.

\begin{conjecture}[Thurston, 1982~\cite{thur3}]\label{conjgeometr}
Let $\Ma$ be a closed, orientable and irreducible 3-manifold. Then,
the geometric decomposition of $\Ma$ is in such a way that
each resulting component is geometric.
\end{conjecture}

The geometrization conjecture (which implies the Poincaré's 
conjecture, as we will see later on this section) was proved by
G. Perelman on a series of articles~\cite{per1, per2, per3} posted
on the ArXiv but never officially published. For proving
the Geometrization conjecture (and, consequently, 
the Poincaré's conjecture), Perelman was awarded with a Fields Medal
in 2006 and with a one million dollars prize given by the
Clay Institute for Mathematics for solving one of the
so-called Millennium Problems. He refused both prizes,
and later he explained
\begin{center}
{\em The Fields Medal was completely irrelevant for me. Everybody
understood that if the proof is correct then no other recognition is needed.}
\end{center}
After Perelman's refusal on the prize, the Clay Institute used 
the one million dollars dedicated for the prize
to fund the {\em Poincaré Chair}
at the Paris Institut Henri Poincaré.
The arguments that Perelman used in order to prove the geometrization
conjecture were deeply analytic, based on the program proposed by 
R. Hamilton~\cite{ham1, ham2, ham3} using the Ricci Flow.
A more detailed version of Perelman's arguments can be found in 
articles such as B. Kleiner and J. Lott~\cite{KleiLott} or
in the monographs by J. Morgan and G. Tian~\cite{MorganT1,MorganT2}.

Although Thurston was not able to prove Conjecture~\ref{conjgeometr}
in its full generality, his work
completely revolutionized 3-dimensional topology. 
First, we mention his classification of
all possible maximal geometries which admit
compact quotients.

\begin{theorem}[Thurston~\cite{thur2}]\label{8geom}
Let $(X,G)$ be a model geometry that admits a compact quotient.
Then, $(X,G)$ is equivalent to one of the eight
geometries $(\bbX,{\rm ISO}(\bbX))$, where $\bbX$ is one 
of the following Riemannian manifolds:
\begin{equation}
\R^{3},\, \hn3,\, \sn3,\,\hn2\times\R,\,\sn2\times\R,
\wt{{\rm SL}}_{2}(\R),\, {\rm Nil}_{3},\, {\rm Sol}_{3}.
\end{equation}
\end{theorem}
\begin{remark}
The geometries given by $\R^{3},\, \hn3,\,
\sn3,\,\hn2\times\R$ and $\sn2\times\R$ are well-known. 
For completeness, we will briefly introduce
the geometries
$\wt{{\rm SL}}_{2}(\R),\, {\rm Nil}_{3},\, {\rm Sol}_{3}$.
They are all defined as 
Lie groups endowed with a left-invariant metric.
\begin{itemize}
\item $\wt{{\rm SL}}(2,\R)$ is the universal covering of the 
special linear group ${\rm SL}(2,\R)$, which is the group of
$2\times2$ real matrices with determinant equal to 1. 
The family of non-isometric 
left invariant metrics on $\wt{{\rm SL}}(2,\R)$ 
has three parameters (which are the three nonzero 
{\em structure constants} when we regard $\wt{{\rm SL}}(2,\R)$ 
as a unimodular Lie group, for more details 
see~\cite[Section~2.7]{MPbook}). Inside this family, there is a 
two-parameter family of metrics for which the isometry group
has dimension four, and when we think as $\wt{{\rm SL}}(2,\R)$
as a model geometry, we think it is endowed with any metric
in this family (for a general metric, the isometry group will have
dimension three and will not be maximal).
\item The Lie group ${\rm Nil}_3$ is easily defined as
the group of upper triangular $3\times 3$ real matrices with diagonal
entries equal to one:
$${\rm Nil}_3 = \left\{\left(
\begin{array}{ccc}
1&x&z\\0&1&y\\0&0&1
\end{array}\right)\mid x,y,z\in \R\right\}.$$
Up to homotheties, there is only one left-invariant metric on 
${\rm Nil}_3$, and its isometry group has dimension four.
\item The {\em solvable group} ${\rm Sol}_3$
is defined as a matrix group as
$${\rm Sol}_3 = \left\{\left(
\begin{array}{ccc}
e^{z} & 0 & x\\
0 & e^{-z} & y \\
0 & 0 & 1
\end{array}\right)\mid x,y,z\in \R\right\}.$$
${\rm Sol}_3$ is the {\em least symmetric} model
geometry of all, since any left-invariant metric 
(it admits a 2-parameter family of them) on it gives rise to a 
3-dimensional isometry group. The left-invariant metrics that
makes this geometry maximal are the ones
that admit some planar reflections, where the
full isometry group has 8 connected components.
\end{itemize}
\end{remark}

The eight model geometries that appear in Theorem~\ref{8geom}
are presently known as {\em Thurston's geometries}. Note that
there are infinitely many non-equivalent model geometries
that do not admit any compact quotient, but, as we will see next,
they are not relevant for geometrization of 3-manifolds.

Other than classifying 3-dimensional model geometries that 
admit compact quotients, Thurston was
able to prove Conjecture~\ref{conjgeometr} for
the following class of manifolds.

\begin{definition}[Haken manifold]\label{defHaken}
A compact, orientable 3-manifold
$\Ma$
is called {\em Haken} if it is irreducible and it
contains an embedded, two-sided, incompressible surface $\S$ of 
genus $g\geq1$.
\end{definition}

A simple consequence of Definition~\ref{defHaken} is that
if $\Ma$ is a compact, orientable and irreducible manifold 
with toroidal boundary, it is Haken. In particular, 
although geometrization is a result for closed 3-manifolds,
a great part of Thurston's work was dedicated on 
understanding geometrization of noncompact manifolds, seen as the
interior of compact manifolds with toroidal boundaries.
The geometrization theorem proved by Thurston can be stated 
as follows.

\begin{theorem}[Thurston's geometrization theorem]\label{thmthgeom}
Let $\Ma$ be a compact, orientable, Haken 3-manifold with either 
empty or toroidal boundary that is not diffeomorphic to
$\bbD\times\sn1,\,\bbT^2\times[0,1]$ or to $\bbK\wt{\times}[0,1]$.
Then, $\Ma$ admits a geometric decomposition such
that each resulting component is geometric and modelled
by one of the eight model geometries that admit a compact quotient.
\end{theorem}

The three exceptions in Theorem~\ref{thmthgeom}
are necessary for two reasons. 
First, they do not appear as components in the geometric
decomposition of any closed, orientable 3-manifold. Secondly, 
they are not geometric in the sense that any complete 
geometric structure in their interiors is of infinite volume.
Also, if a closed manifold $\Ma$ has finite fundamental group,
then it is never Haken, so Theorem~\ref{thmthgeom}
does not apply to them (and does not help to solve the Poincaré 
conjecture). However, it is worthwhile mentioning that
Thurston's geometrization theorem (and his geometrization conjecture)
put the Poincaré conjecture in a broad perspective, and changed
the point of view that the topologist community had on the subject.
We next quote the words of John Morgan, in his talk at the 2018 
Clay Research Conference.

\begin{center}
{\em 
[Before Thurston's work,] there was no strong reason to believe that 
Poincaré's conjecture was either true or false.
[...] But the fact that you put the Poincaré conjecture, which was 
about one particular 3-manifold, in a vast conjecture that is supposed 
to classify all 3-manifolds, and has some positive evidence for it, 
makes you believe that you shouldn't spend your time looking for a 
counterexample.}
\end{center}

As we already mentioned, 
the full geometrization theorem was proved by Perelman,
and its statement, unifying both
closed (Conjecture~\ref{conjgeometr}) 
and compact 3-manifolds with toroidal boundary 
(Theorem~\ref{thmthgeom}) in the same result,
can be read as follows.

\begin{theorem}[Geometrization theorem, Perelman]\label{thmPergeom}
Let $\Ma$ be a compact, orientable, irreducible 3-manifold with either 
empty or toroidal boundary that is not diffeomorphic to
$\bbD\times\sn1,\,\bbT^2\times[0,1]$ or to $\bbK\wt{\times}[0,1]$.
Then, we can cut $\Ma$ along a finite, possibly empty collection
of incompressible, disjointly embedded
surfaces, each of which either is a torus or a Klein bottle,
such that each resulting component is geometric and modelled
by one of the eight Thurston's geometries.
\end{theorem}

While Thurston's proof of Theorem~\ref{thmthgeom} is mostly
topological,
the proof of Theorem~\ref{thmPergeom}
when $\Ma$ is a closed, non-Haken manifold carried out by 
Perelman is radically different, being deeply analytic.
Both proofs are out of the scope
of this manuscript, but we next present an intuitive and superficial
account of Perelman's proof.
This approach was first suggested 
by Richard Hamilton, which 
defined the Ricci flow and thought that it 
could be used to prove
the geometrization conjecture. This flow may intuitively
be regarded as a heat flow for manifolds, but instead of
distributing temperature uniformly,
it changes the metric of a manifold, distributing
curvature uniformly, hopefully
converging to a stationary state where the underlying
smooth manifold now has a geometric metric.
The main difficulty on this argument 
is that the Ricci flow generates
certain singularities in finite time,
preventing one to obtain a geometric limit,
but Hamilton visualized that the singularities, if controlled,
could actually provide the geometric decomposition of the 
original manifold\footnote{In fact, the geometric decomposition is not
provided by the singularities themselves, but by components that
even after rescaling
are collapsing, in the Gromov-Hausdorff sense,
to a lower dimensional space with curvature bounded from 
below by $-1$.}. 
Perelman, after a
deep analysis that (under some technical assumptions) 
classified all possible singularities
on the Ricci flow, could then 
perform the {\em Ricci flow with surgeries},
controlling the topology of the original 
manifold where a singularity appeared and continuing the flow
past this singularity, doing so only a finite number of times and
obtaining a convergence as Hamilton envisaged.

\

When a closed 3-manifold $\Ma$ is 
geometrizable, its underlying geometry
is unique (see~\cite[Theorem~5.2]{scott}). Moreover,
it is possible
to determine the underlying geometry of $\Ma$
in terms of its topology.
The next two theorems
provide this description. The first result deals with the 
6 geometries that give rise to Seifert fibered spaces, and classify
the geometry in terms of the Euler characteristic $\chi$ of
the base orbifold and the Euler number $e$ of the Seifert fibration,
while the second treats with the ${\rm Sol}_3$ geometry.

\begin{theorem}[{P. Scott~\cite[Theorem~5.3~(ii)]{scott}}] \label{classif}
A closed 3-manifold $\Ma$ 
admits a geometric structure modelled
on $\R^3,\,\sn3,\,\sn2\times\R,\,\hn2\times\R,\wt{{\rm SL}}(2,\R)$ or
${\rm Nil}_3$ if and only if
$\Ma$ is a Seifert fibered space. In this case,
the relation between the underlying geometry of $\Ma$ and the
Euler number $e$ and the Euler characteristic $\chi$ of the Seifert
fibration is given by the following table.
\begin{center}
\begin{tabular}{ c|c c c }
  & $\chi <0$ & $\chi = 0$ & $\chi > 0$  \\  [1ex] 
\hline \\ [-1ex] 
$e=0$ & $\Hy^{2} \times \R$ &  $\R^{3}$ & $\Sph^{2} \times \R $\\  [1ex]
 $e \neq 0 $ & $\widetilde{SL}_{2}( \R )$ & $Nil_{3}$ & $\Sph^{3}    $
\end{tabular}
\end{center}
\end{theorem}

\begin{theorem}[{P. Scott~\cite[Theorem~5.3~(i)]{scott}}]\label{thmSolg}
A closed 3-manifold
$\Ma$ admits a geometric structure modelled in ${\rm Sol}_3$
if and only if $\Ma$ is finitely covered by 
a {\em torus bundle}\footnote{Let $\bbT^2$ be the 2-torus and
let $f: \bbT^{2} \to \bbT^{2}$ 
be a orientation preserving homeomorphism.
The torus bundle generated by $f$ is the 3-manifold
$M(f) = (\bbT^2\times [0,1])/\sim$, where $\sim$ is the identification
$\bbT^2\times\{0\}\ni (x,0) \sim (f(x),1) \in \bbT^2\times\{1\}$.}
over $\sn1$ with hyperbolic identification map.
\end{theorem}

Together, Theorems~\ref{classif} and~\ref{thmSolg} provide a complete
classification of a closed,
geometric 3-manifold $\Ma$ in terms of its topology: if $\Ma$
is a Seifert-fibered space, the underlying geometry is one of
the {\em Seifert fibered geometries} of Theorem~\ref{classif}.
If $\Ma$ is finitely covered by a torus bundle with hyperbolic
identification map, it is modelled by ${\rm Sol}_3$. Otherwise,
the underlying geometry is $\hn3$.

Another result that classify the underlying geometry of a 
geometric 3-manifold in terms of its topology (in some sense
extending Theorems~\ref{classif} and~\ref{thmSolg})
is the following Theorem~\ref{geompi1}, which can be found 
in~\cite[Section~1.8]{3mg} or as~\cite[Proposition~12.8.3]{martelli}.
It makes use of the following nomenclature:
if $P$ is a certain group property, we say that a 
group $G$ is {\em virtually $P$} if $G$ admits a finite index subgroup
that satisfies the property $P$.

\begin{theorem}\label{geompi1}
Let $\Ma$ be a closed, orientable
3-manifold modelled by a geometry $\mathbb{X}$.
Then, the following hold:

\begin{itemize}
\item
If $\pi_1(\Ma)$ is finite, then $\bbX = \sn{3}$. In this case, $\Ma$ is 
finitely covered by $\sn{3}$ and has a structure of
a Seifert fibered space with $\chi>0$ and $e\neq 0$.
\item Otherwise, if $\pi_{1}(\Ma) = \Z$ or $D_{\infty}$, then
$\bbX = \sn2\times\R$. In this case, $\Ma$ is
either $\sn2\times \sn1$ or $\R\mathbb{P}^3\#\R\mathbb{P}^3$,
so it has a structure of a Seifert fibered space with 
$\chi >0$ and $e = 0$.
\item Otherwise, if $\pi_{1}(\Ma)$ is virtually $\Z^{3}$, then
$\bbX = \R^{3}$. In this case, $\Ma$ is finitely covered by 
$T^3=\sn1\times\sn1\times\sn1$ and has a structure of a Seifert fibered
space with $\chi = 0$ and $e = 0$.
\item Otherwise, if $\pi_{1}(\Ma)$ is virtually nilpotent, then 
$\bbX = {\rm Nil}_{3}$. In this case, $\Ma$ is finitely covered by
a torus bundle with nilpotent monodromy, and has a Seifert
fibered structure with $\chi=0$ and $e \neq 0$.
\item Otherwise, if $\pi_{1}(\Ma)$ is solvable, then 
$\bbX = {\rm Sol}_{3}$.
In this case, $\Ma$ (or a double cover of $\Ma$) is a torus bundle with
Anosov monodromy and $\Ma$ does not admit a Seifert fibered structure.
\item Otherwise, if $\pi_{1}(\Ma)$ is virtually a product 
$\Z \times F$, where $F$ is noncyclic and free, then $\bbX = \hn2\times\R$. 
In this case, $\Ma$ is finitely covered by $\S\times \sn1$, $\S$ a 
surface with $\chi(\Sigma)<0$ and $\Ma$ admits a Seifert fibered 
structure with $\chi<0$ and $e = 0$.
\item Otherwise, if $\pi_{1}(\Ma)$ is a nonseparable extension 
of a noncyclic free group $F$ by $\Z$, then
$\bbX= \widetilde{{\rm SL}}(2,\R)$. In this case,
$\Ma$ is finitely covered by an $\sn1$ bundle over a surface 
$\Sigma$ with $\chi(\Sigma) <0$ and $\Ma$ admits a Seifert
fibered structure with $\chi <0$ and $e\neq 0$.
\item Otherwise, then $\bbX = \hn3$. In this case, $\Ma$ is atoroidal
and does not admit a Seifert fibered structure.
\end{itemize}
\end{theorem}

An important observation is that when we are classifying the underlying
geometry of a geometric 3-manifold, the hyperbolic case is always the
{\em otherwise} case. In fact, the two most difficult steps into
proving geometrization in its full generality were the 
{\em elliptization conjecture} and the {\em hyperbolization
conjecture}, that dealt with the respective $\sn3$ and $\hn3$ 
geometries. 

\begin{theorem}[Elliptization Theorem]\label{ellipthm}
Let $\Ma$ be a closed, orientable 3-manifold with 
finite fundamental group. Then $\Ma$ is {\em elliptic}, 
i.e., $\Ma$ admits a geometric structure
modelled by the 3-sphere $\sn3$.
\end{theorem}

\begin{theorem}[Hyperbolization Theorem]\label{hypthm}
Let $\Ma$ be a compact, orientable and irreducible 3-manifold
with (possibly empty) toroidal boundary,
$M\neq \bbD\times\sn1,
\,M\neq \bbT^2\times[0,1],\,M\neq \bbK\wt{\times}[0,1]$. 
If $\Ma$ is atoroidal and
$\pi_1(\Ma)$ is infinite, then $\Ma$ is {\em hyperbolic}, i.e.,
$\Ma$ admits a geometric structure of finite volume
modelled by $\hn3$.
\end{theorem}

As already mentioned, Thurston proved the geometrization in the
class of Haken manifolds, and the most crucial step
in his proof
was to prove Theorem~\ref{hypthm} for this class of manifolds.
Next, we will show how both
the hyperbolization and elliptization theorems follow from
the geometrization theorem.

\begin{proof}[Sketch of the proof of Theorems~\ref{ellipthm} 
and~\ref{hypthm}]
Let $\Ma$ be a compact, orientable 3-manifold 
that satisfies either the hypothesis of Theorem~\ref{ellipthm} or
of Theorem~\ref{hypthm}. Note that if $\Ma$ is closed and
$\pi_1(\Ma)$ is finite, $\Ma$ is irreducible, and in both cases
there is no $\Z\times\Z$ subgroup
in $\pi_1(\Ma)$, so
the JSJ decomposition of $\Ma$, given by 
Theorem~\ref{jsj}, must be trivial. 
Hence, Theorem~\ref{thmPergeom} gives that
$\Ma$ is itself geometric and admits a model geometry $\bbE$. 
The fact that $\pi_1(\Ma)$ does not contain any $\Z\times\Z$ subgroup
implies directly that $\bbE$ is not one of $\wt{{\rm SL}}(2,\R),\,
\R^3,\,\hn2\times\R,\,{\rm Nil}_3$ and ${\rm Sol}_3$ (whose quotients
always have such a subgroup). Hence, either $\bbE = \sn3$ or 
$\bbE = \hn3$. Since a quotient of $\sn3$ has finite fundamental
group while no (nontrivial) quotient of $\hn3$ does so,
this proves both theorems.
\end{proof}

At this point, it is almost irrelevant
to present (or prove) the next statement. However, due
to its beautiful and old history, we chose to do so.

\begin{corollary} The Poincaré conjecture is true.
\end{corollary}
\begin{proof}
Let $\Ma$ be a simply connected, closed, orientable 3-manifold.
Then, Theorem~\ref{ellipthm} implies that the total space of its
universal covering is $\sn3$. But since $\pi_1(\Ma)$ is trivial,
the covering map $\pi\colon \sn3\to\Ma$ must be a diffeomorphism.
\end{proof}

To finish this section, we note that the Geometrization
Theorem has several deep applications not only to topology but
also to differential geometry, and next we present just a few of them,
in order to illustrate the advances it made possible.
First, we observe that geometrization allowed for the solution
of the so called {\em homeomorphism problem}, which asks for an 
algorithm to decide if two given compact
3-manifolds are homeomorphic, for details 
see~\cite[Section 1.4.1]{BBetal}. Also the complete proof
of the Poincaré conjecture, together with previous 
results by Gromov-Lawson and Schoen-Yau, made possible to obtain the
topological 
classification of closed 3-manifolds that admit a metric of
positive scalar curvature (see, for instance,~\cite{CaLi}
or~\cite{KleiLott}). 
Moreover, it also was used in the classification
of 3-manifolds with non-negative Ricci curvature~\cite{Liu}.

\section{Hyperbolization of noncompact 3-manifolds}\label{sechyperb}

As seen in Section~\ref{sec3manf}, the richest topology of all
is the one of the hyperbolic 3-manifolds, and a great
part in the work of Thurston was to obtain a deep understanding of
the topology of such manifolds. In this section, we will focus
our attention to the hyperbolization of noncompact 3-manifolds.
We will present an algorithmic characterization,
equivalent to the hyperbolization theorem, that allows us to
decide whether the interior of a compact, orientable 3-manifold
with nonempty toroidal boundary admits a complete hyperbolic metric
of finite volume. We will also present a few applications 
of this criterion. We would like to point out that 
several recent developments in the theory of hyperbolic 3-manifolds
were achieved by the works of Agol, Kahn, Markovic, Wise and many
others and we suggest the book~\cite{3mg} by Aschenbrenner, Friedl and
Wilton and its extensive list
of references for aspects of hyperbolic 3-manifolds not covered in 
this manuscript. We also suggest the book~\cite{BPHyp} by Benedetti and
Petronio for several classical results for hyperbolic 3-manifolds.

\subsection{Thurston's hyperbolicity criterion}

Although presently we know that among the geometries
appearing on the geometric decomposition of closed 3-manifolds, the
hyperbolic geometry is the most prominent,
until the work of Thurston very few explicit examples of 
noncompact hyperbolic 3-manifolds of finite volume were 
known.
In the words of Thurston~\cite{thurletter}, we quote:

\begin{center}{\em
To most
topologists at the time, hyperbolic geometry was 
an arcane side branch of mathematics, although there were other 
groups of mathematicians such as differential
geometers who did understand it from certain points of view.}
\end{center}

We start this section
by stating a celebrated result in knot theory, proved by Thurston,
that provided infinitely many
examples of hyperbolic 3-manifolds as complements of knots in $\sn3$ 
(recall that a {\em knot} in a manifold $P$ is 
the image of an embedding $f\colon \sn1\to P$).
Along this section, by a {\em hyperbolic 3-manifold} we mean
a noncompact, orientable 3-manifold endowed with 
a complete hyperbolic metric of finite volume.
Also, when $K$ is a subset of a 3-manifold, $N(K)$
will denote a small, open, regular tubular neighborhood around $K$
and $\ol{N}(K)$ will denote its closure, when these concepts 
make sense.

\begin{theorem}[Classification of knots in $\sn3$]\label{knotss3}
Let $K$ be a knot
in $\sn3$. Then, one of the following holds (see 
Figure~\ref{figknots}):
\begin{itemize}
\item $K$ is a {\em torus knot}, i.e., there exists an ambient isotopy
that maps $K$ to the boundary of a standard\footnote{Using 
the model $\sn3\subset\R^4$, a {\em standard} torus (or 
a unknotted torus) is 
any torus that is isotopic to $\{(x,y,z,w)\mid x^2+y^2=\frac12,\,z^2+y^2=\frac12\}\sim \sn1\times\sn1$.} 
solid torus in $\sn3$;
\item $K$ is a {\em satellite knot}, i.e., there exists a knot
$\wt{K}$ in $\sn3$ such that $K$ is contained in a regular, tubular 
neighborhood around $\wt{K}$, and $K$ is not isotopic to $\wt{K}$;
\item $K$ is a {\em hyperbolic knot}, i.e., the open manifold
$\sn3\setminus \ol{N}(K)$
admits a complete hyperbolic metric of finite volume. 
\end{itemize}
\begin{figure} 	
\centering
	\includegraphics[scale=0.47]{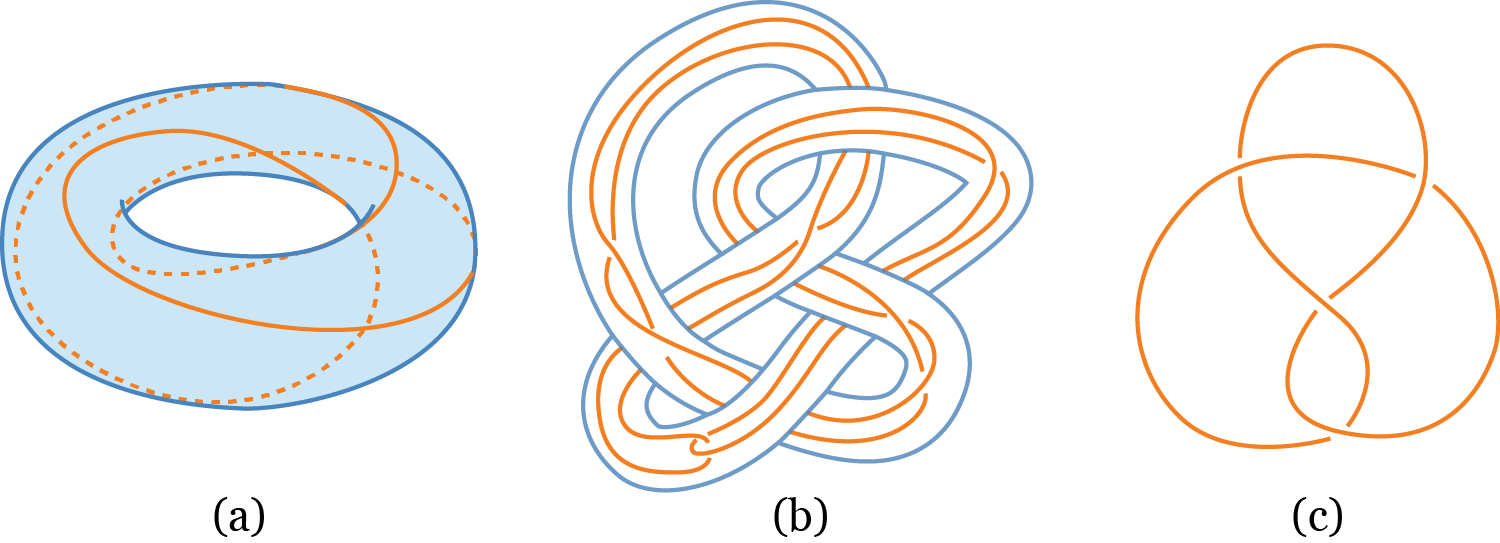}
\caption{(a) The torus knot $(2,3)$; (b) A satellite knot of the 
knot $5_{2}$; (c) An example of a hyperbolic knot, $4_{1}$\label{figknots}}
\end{figure}

\end{theorem}

\begin{remark}\label{remarkColin}
The first known
example of a noncompact, complete 
hyperbolic 3-manifold of finite volume was given by 
H. Gieseking~\cite{Giese} in 1912. 
Such a manifold is actually nonorientable and 
C. Adams~\cite{Ad1} proved that 
it is the noncompact hyperbolic
3-manifold with the smallest possible volume\footnote{Among several results about hyperbolic 3-manifolds
of finite volume we are omitting in this article is the {\em 
Mostow-Prasad rigidity theorem}, which states that the geometry
of hyperbolic 3-manifolds is rigid. In contrast with the case of 
hyperbolic surfaces, where the same topology may admit infinitely many
hyperbolic metrics, in dimension three
any diffeomorphism between two hyperbolic 3-manifolds
is isotopic to an isometry. Thus, the {\em hyperbolic volume} of
a given hyperbolic 3-manifold is a well-defined topological
invariant.}
$\mathcal{V}\simeq 1.0149$, where $\mathcal{V}$ is 
the volume of the ideal regular tetrahedron 
in $\hn3$. 
The first description of a hyperbolic 3-manifold 
as a knot complement in $\sn3$ is due to R. Riley, and it is the 
complement of the {\em Figure-eight knot}, which 
is the double cover of the Gieseking manifold. It was proven by
Cao and Meyerhoff~\cite{CaoMeyer} that the Figure-eight knot complement
and
its {\em sibling manifold} (which is not a knot complement 
in $\sn3$ but
can be described as 
$(5,1)$ Dehn surgery on the right-handed Whitehead Link),
are the two unique orientable, noncompact hyperbolic 3-manifolds with
the minimum volume $2\mathcal{V} \simeq 2.0298$.
\end{remark}

Theorem~\ref{knotss3} provides a great intuition of how to find
noncompact hyperbolic 3-manifolds. Let $P$ be a closed 3-manifold and
let $L$ be a link in $P$ (i.e. a finite, pairwise disjoint collection
of knots). Then, the manifold $\Ma = P\setminus N(L)$
is a compact 3-manifold with toroidal boundary.
Each boundary component comes from a component of the
original link $L$ and gives rise to an {\em end} (with the
topology of $\bbT^2\times[0,\infty)$) of
the open manifold $P\setminus \ol{N}(L)$. 
Moreover, Theorem~\ref{hypthm} gives us the
intuition that if the link $L$ is sufficiently
{\em complicated}, the interior of $\Ma$ will admit a complete 
hyperbolic metric of finite volume. 
This intuition will be put in a rigorous form 
in Theorem~\ref{HypCrit} below, right after a definition 
necessary for its statement. Here, we
will let $\Ma$ denote
a connected, orientable, compact 3-manifold with toroidal boundary
and, once again, we will use the nomenclature
{\em a surface $\S$ in $\Ma$} to represent
a properly embedded surface $\S\subset \Ma$, i.e., $\S$ is compact, 
embedded in $\Ma$ and $\partial \S = \S \cap \partial \Ma$.

\begin{definition}\label{DefVarias} \
\begin{enumerate}
\item A sphere $S$ in $\Ma$ is {\em essential} if $S$ does not
bound a ball in $\Ma$. If $\Ma$ does not admit
any essential sphere, $\Ma$ is {\em irreducible} (as in 
Definition~\ref{defirred}).
\item A disk $D$ in $\Ma$ is {\em essential}
if $\partial D$ is homotopically nontrivial in $\partial \Ma$.
If $\Ma$ does not admit any essential disks, $\Ma$ is called
{\em boundary-irreducible}.
\item A torus $T$ in $\Ma$ is {\em essential}
if $T$ is incompressible (as in Definition~\ref{incompr})
and not boundary parallel, in the sense that it is not 
isotopic to a component 
of $\partial \Ma$. If $\Ma$ does not admit any essential torus,
$\Ma$ is {\em geometrically atoroidal} (as in Definition~\ref{defgam}).
\item An annulus $A$ is {\em essential}
in $\Ma$ if $A$ is incompressible, 
boundary-incompressible\footnote{If $\S$ is a surface in $\Ma$, a
{\em boundary-compression disk} for
$\S$ is a disk $D$ in $\Ma$ with
$\partial D = D\cap (\S \cup \partial \Ma)$
such that $\partial D = \alpha \cup \beta$, where $\alpha= D \cap \S$
and $\beta = D \cap \partial \Ma$
are arcs intersecting only
in their endpoints, and $\alpha$ does
not cut a disk from $\S$. If $\S$ does not admit
any boundary-compression disk, we say that $\S$ is
{\em boundary-incompressible}.}
and not boundary parallel, in the sense that it is not
isotopic, relative to $\partial A$,
to an annulus $A'\subset \partial \Ma$.
If $\Ma$ does not admit any essential annuli, we say $\Ma$ is
{\em acylindrical}.
\item If $\Ma$ is irreducible, boundary-irreducible, geometrically
atoroidal and acylindrical, we say that $\Ma$ is {\em simple}.
\end{enumerate}
\end{definition}

The above definition allows us to obtain a criterion for
hyperbolicity that generates noncompact, complete hyperbolic 
3-manifolds of finite volume. It follows directly from
Theorem~\ref{hypthm}, so it is also commonly known as
the {\em Hyperbolization theorem}.

\begin{theorem}[Thurston's hyperbolization criterion]\label{HypCrit}
Let $\Ma$ be an orientable, compact 3-manifold with nonempty
toroidal boundary. Then, $\Ma$ is hyperbolic if and only if $\Ma$
is simple.
\end{theorem}

\begin{proof}
Let $\Ma$ be as stated and assume that $\Ma$ is simple.
We will show that $\Ma$ satisfies the hypothesis of 
Theorem~\ref{hypthm}. First, $\Ma$ is
irreducible by the definition of being simple. Also, 
since $\partial \Ma$
is nonempty and toroidal, $\pi_1(\Ma)$ is infinite.
Furthermore, 
the fact that there are no essential disks implies that
$\Ma\neq \bbD\times\sn1$.
Since both $\bbT^2\times[0,1]$ and $\bbK\wt{\times} [0,1]$ contain
essential annuli, $\Ma$ is neither of them.
It remains to show that $\Ma$ is atoroidal. By hypothesis,
$\Ma$ is geometrically atoroidal, so the other option 
(see Remark~\ref{smallseifert}) is that 
$\Ma$ is a small Seifert fibered space. But since 
$\partial \Ma \neq \emptyset$, we may use Proposition~\ref{propClass}
to see that if the number of components in $\partial \Ma$
is one, then $\Ma$ admits an essential disk and if it is two or three,
it admits an essential annulus, a contradiction since $\Ma$ is simple.

On the other hand, assume that $\Ma$ is hyperbolic.
Then, there exists a complete hyperbolic metric of finite
volume in ${\rm int}(\Ma)$ and a standard minimization argument
(such as in~\cite{HRW} or in~\cite{CHMR}) 
shows that any essential sphere, disk, torus or annulus in $\Ma$
would provide a properly embedded minimal surface in the
hyperbolic metric of ${\rm int}(\Ma)$ with nonnegative Euler
characteristic. Since
there are no such a surface in a hyperbolic 3-manifold of finite
volume (see, for instance,~\cite[Theorem~2.1]{CHR} or 
\cite[Corollary~4.7]{meramos}), $\Ma$ is simple.
\end{proof}

To finish this section, we will present some examples of
link complements in $\sn3$. Sometimes, we shall talk about
{\em projections} of links, which is an intuitive concept, but
we suggest~\cite{Adbook} as a first reference to the reader interested
in the topic.
As in Theorem~\ref{knotss3}, we say that 
a link $L$ in $\sn3$ is hyperbolic if $\sn3\setminus N(L)$ admits
a complete hyperbolic metric of finite volume.

\begin{example}[The unknot, Figure~\ref{figtrefoil}~(a)]
Let $K\subset \sn3$ be the trivial knot (i.e., $K$ is isotopic in $\sn3$
to $\{(x,y,z,w)\in\sn3\mid x^2+y^2 = 1, z= w = 0\}$).
Via stereographic projection, we may see that $\sn3\setminus K$ is 
diffeomorphic to $\R^3\setminus Z$, where $Z$ is the $z$-axis.
Since $\R^3\setminus Z$ admits a product structure 
$\sn1\times(0,\infty)\times \R$, we can see that $\sn3\setminus K$
is diffeomorphic to $\sn1\times \R^2$ and to $\sn1\times \bbD$.
In particular, using the coordinates of $\sn1\times\bbD$, we may see
that for any given $p\in \sn1$, $\{p\}\times\bbD$ 
is an essential disk in $\sn3\setminus K$, hence
$K$ is not hyperbolic.
\end{example}

\begin{example}[The trefoil knot, Figure~\ref{figtrefoil}~(b)]
If $K$ is the trefoil knot,
$K$ is a torus knot, hence it is not hyperbolic. 
Note that $\sn3\setminus N(K)$ admits an essential annulus.
The complement of the trefoil knot was among the first 3-manifolds
(since the trefoil is the {\em simplest} nontrivial knot, in the
sense that it has the projection with the fewest possible crossings)
which Thurston attempted to endow with a complete hyperbolic metric
of finite volume, before he developed his criterion. He didn't succeeded
because it was not possible, although he still wasn't aware of that.
\end{example}

\begin{example}[The figure-eight knot, Figure~\ref{figtrefoil}~(c)]
If $K$ is the figure-eight knot,
$\sn3\setminus N(K)$ is hyperbolic
and its complement has
a hyperbolic volume of approximately 2.0298.
Although its hyperbolicity was proved first by R. Riley, 
this was the first noncompact 3-manifold in which Thurston could find
a hyperbolic structure of finite volume.
\end{example}

\begin{example}[The Hopf link, Figure~\ref{figtrefoil}~(d)]
The Hopf link $L$ is the union of two trivial knots 
$C_1,\,C_2$.
It is not hyperbolic, since its complement admits an essential annulus.
Using the stereographic projection about a point (say in $C_1$) turns
$\sn3\setminus C_1$ diffeomorphically into $\R^3\setminus Z$, where
$Z$ denotes the $z-$axis. In particular, $\sn3\setminus L$ is
diffeomorphic to $\R^3\setminus (Z \cup C)$, where 
$C = \{(x,y,0)\in\R^3\mid x^2+y^2 = 1\}$, which is
easily seen as diffeomorphic to $\bbT^2\times (0,\infty)$
since it admits a foliation by tori having $C$ as their core curves.
In particular, the manifold $\sn3\setminus N(L)$ is diffeomorphic
to $\bbT^2\times[0,1]$, which can also be seen as non-hyperbolic, 
but this time we go further and notice it admits not
one (or two, which are easy to find in $\sn3\setminus L$) but
infinitely many non-isotopic essential annuli, just take
a nontrivial curve $\gamma$ in $\bbT^2$ and look at the
annulus $\gamma\times[0,1]$ in $\bbT^2\times[0,1]$.
\end{example}

\begin{example}[Borromean rings, Figure~\ref{figtrefoil}~(e)]
The Borromean rings with three components are three 
disjoint, trivially embedded circles such that each two of them are not
linked (in the sense that there exists a sphere which separates
one from the other), but the three components together are linked.
It is hyperbolic, and its complement has a hyperbolic volume 
of approximately 7.3277.
\end{example}

\begin{example}[The $(n,s)$-chain, 
Figure~\ref{figtrefoil}~(f)]\label{egdaisy}
For given $n\geq3$ and $s\in \Z$, a $(n,s)$-chain is a link 
with $n$ trivial components
$C_1,\,C_2,\,\ldots,\,C_n$ in such a way that for each 
$i\in\{1,\,\ldots,\,n\}$ the
component $C_i$ is linked only with $C_{i-1}$ and with $C_{i+1}$ 
(where we extend our notation to allow 
$C_0 = C_n$ and $C_{n+1} = C_1$), and 
each pair $C_i$ and $C_{i+1}$ is linked as in the Hopf link. Also,
we add $s$ left half twists to one of the components (if $s\geq0$, the 
link is alternating, otherwise we add $-s$ right twists to
one component and the projection of the
link will no longer be alternating). 
It was proven by W. Neumann and A. Reid~\cite{NeuR}
that the $(n,s)$-chain
is hyperbolic if and only if 
$\{\abs{n+s},\abs{s}\} \not\subset \{0,1,2\}$.
In particular, a chain with 3 components is not hyperbolic if and
only if $s = -1$ or $s = -2$, a chain with 4 components is 
hyperbolic unless $s = -2$ and any chain with 5 or more
components is hyperbolic.
\end{example}

\begin{example}[Composite knots, Figure~\ref{figtrefoil}~(g)]
There is a notion of composition for oriented 
knots,
and a knot is called a {\em composite} 
knot if it can be obtained from such an operation.
A knot is called {\em prime} if it cannot be obtained from two
nontrivial knots by composition. 
There is a simple characterization to identify if 
a knot $K$ in $\sn3$ is a composite knot. Let $S$ be an embedded
sphere in $\sn3$, separating $\sn3$ into two ball regions 
$B_1,\,B_2$. If $S\cap K$ consists of a transversal intersection
in two points and the resulting components $\gamma_1 = K\cap B_1$ and
$\gamma_2 = K\cap B_2$ are nontrivial, 
in the sense that there is no isotopy in $B_i$ that fixes
the endpoints of $\gamma_i$ and
maps $\gamma_i$ to $S$ (in other words, the two pieces of $K$ in 
each of $B_1,\,B_2$ are themselves knotted while fixing 
their endpoints in $S$), then
$L$ is a composite.
Note that a composite knot is never hyperbolic, since
the sphere $S$ with the two points removed provides an essential
annuls in $\sn3\setminus N(L)$.
\end{example}

\begin{example}[Unlinked knots, Figure~\ref{figtrefoil}~(h)]
Let $K_1,\,K_2$ be two distinct knots in $\sn3$. If there 
is an embedded sphere $S$ that separates 
$K_1$ and $K_2$, then the knots are
not linked. Since, in this case, $S$ is an essential sphere to 
$\sn3\setminus N(K_1\cup K_2)$, the link $K_1 \cup K_2$ is not 
hyperbolic.
\end{example}

\begin{figure}	
\centering
\includegraphics[width=0.9\textwidth]{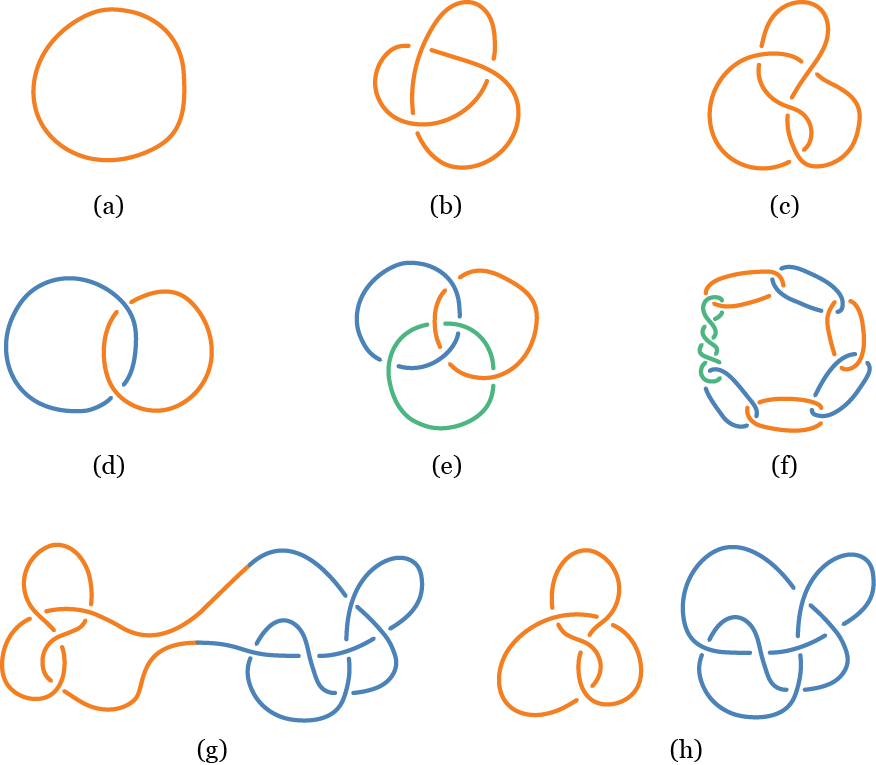}
\caption{(a) The Unknot; (b) The Trefoil knot; 
(c) The Figure-eight knot; (d) the Hopf link; 
(e) The Borromean rings with three components; 
(f) An $(n,s)$-chain with $n = 7$ and $s = 3$; 
(g) A composite knot, obtained by the composition 
of a figure-eight knot and the knot $6_{3}$; 
(h) Two knots (figure-eight knot and $6_{3}$) which are unlinked. 
\label{figtrefoil}}
\end{figure}

A useful tool for deciding whether a given link in $\sn3$ is hyperbolic
or not is the software SnapPy~\cite{Snappy}, which allows the user 
to draw a projection of a knot or a link and computes
several topological invariants (including hyperbolic 
volume, if the link is hyperbolic).
According to the developers,
{\em SnapPy is a program for studying the topology and geometry of 
3-manifolds, with a focus on hyperbolic structures}.
It was written using the kernel of a previous program, SnapPea,
by Jeff Weeks.

\subsection{Applications to knot and link complements}

With a little effort, it is not difficult to obtain
Theorem~\ref{knotss3} from Theorem~\ref{HypCrit}. In fact, any
complement of a torus knot will admit an essential disk (if the knot 
is the unknot) or an essential annulus (the 
projection of the knot to the standard torus separates it into annuli),
and any satellite knot admits an essential torus. In this section, 
we will present some recent results concerning hyperbolicity of 
link complements in 3-manifolds that use Thurston's hyperbolicity
criterion in their proofs. These results
are based on the works of Adams, Meeks and the second 
author~\cite{amr1,amr2} and
allowed the construction
of new examples of hyperbolic 3-manifolds of finite volume,
containing totally umbilic surfaces for any admissible topology and
mean curvature.

The setting to be considered is the following.
Let $P$ be a closed 3-manifold and let $L$ be a link in $P$ such that
the compact manifold with nonempty toroidal boundary 
$M = P\setminus N(L)$ is hyperbolic (this will be assumed throughout
all the statements that follow).
We consider two moves that one can perform on $L$ to obtain
a new link $L'$ in $P$ such that the corresponding manifold
$M' = P\setminus N(L')$ will also be hyperbolic.

The first move we
consider is called the {\em chain move}~\cite[Theorem~3.1]{amr1}.
Here, we start
with a trivial component bounding a twice-punctured disk
in a ball $\cB\subset P$ as in Figure~\ref{chainlemma1}, 
and we replace the
tangle on the left with the tangle on the right
in Figure~\ref{chainlemma1}, where $k$ is any integer. Assuming that
the $(P\setminus\cB)\cap M$ is not
the complement of a rational tangle in a 3-ball (in particular this is
trivially satisfied if $P\neq \sn3$, see~\cite[Chapter~2]{Adbook}),
the result is hyperbolic.
We note that there are counterexamples to extending the result to 
the case where $P = \sn3$ and $(P\setminus\cB)\cap M$ is a
rational tangle complement in a 3-ball, but they are completely 
classified in the appendix of~\cite{amr1}.

\begin{figure}
\centering
\includegraphics[width=0.95\textwidth]{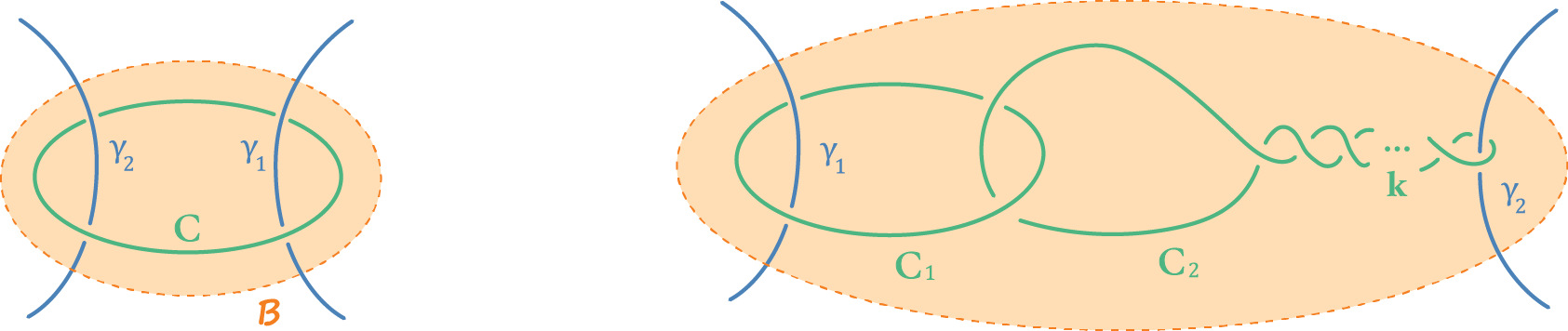}
\caption{Inside the 3-ball $\cB$, replacing the component $C$
by the two components $C_1,\,C_2$ preserves hyperbolicity of the 
complement, provided that $(P\setminus \cB)\cap L$ is not
a rational tangle in a 3-ball.\label{chainlemma1}}
\end{figure}

The second move is called the 
{\em switch move}~\cite[Theorem~4.1]{amr1}. Suppose 
that $\alpha$ is an embedded arc in $P$ that intersects $L$
only on its two distinct endpoints and with interior that
is isotopic (in $P\setminus L$) 
to an embedded geodesic in the hyperbolic metric of $M$.
Let $\cB$ be a small neighborhood
of $\alpha$ in $P$. Then $\cB$ intersects
$L$ in two arcs, as in Figure~\ref{figaugmented}~(a). The switch move
allows us to surger $L$ and
add in a trivial component as in Figure~\ref{figaugmented}~(b) while
preserving hyperbolicity.

\begin{figure}
\centering
\includegraphics[width=0.99\textwidth]{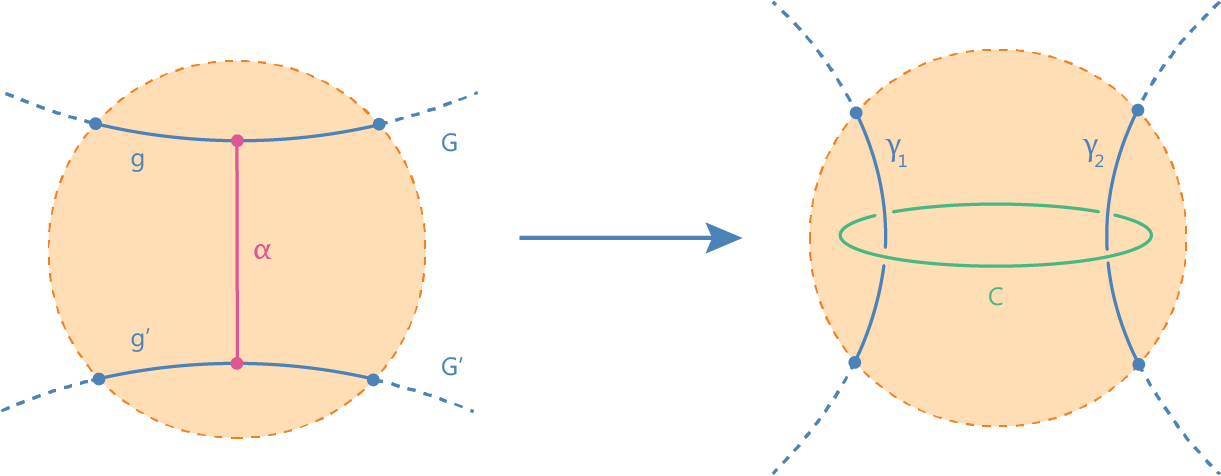}
\caption{The switch move replaces the
arcs $g$ and $g'$ in a neighborhood of a complete geodesic 
by the tangle $\g_1\cup~\g_2\cup C$.}
\label{figaugmented}
\end{figure}

The proofs of the moves above follow step by step Thurston's
criterion presented in Theorem~\ref{HypCrit} above. Namely,
the authors analyze all possibilities for an essential surface
that would prevent $M'$ from being hyperbolic and show that 
whenever such obstruction exists, there is an obstruction 
for the hyperbolicity of $M$ as well. 
Together, they allowed the following construction, which 
is one of the main results of~\cite{amr2}.

\begin{theorem}[{Adams-Meeks-Ramos~\cite[Theorem~1.2]{amr2}}]
\label{thmumbilic}
A connected surface $S$ appears as a
properly embedded totally umbilic surface
with constant mean curvature $H\in[0,1)$
in some hyperbolic 3-manifold of finite volume if and only if
$S$ has finite negative Euler characteristic.
\end{theorem}

\begin{remark}
In fact, Theorem~1.1 of~\cite{amr2}
implies that any embedded totally umbilic surface in a hyperbolic
3-manifold of finite volume is proper. Also, the same theorem
proves that if $\S$ is a totally umbilic
surface properly embedded in a hyperbolic 3-manifold of finite
volume, then $\chi(\S)<0$ if and only if
the mean curvature of $\S$ satisfies
$\abs{H_\S}<1$. Thus, in some sense, Theorem~\ref{thmumbilic} 
is sharp, since it shows that any possible pair of topological type
and mean curvature $H\in [0,1)$ actually exists.
\end{remark}

\begin{proof}[Main steps in the proof of Theorem~\ref{thmumbilic}]
First, we briefly sketch the complete proof of the theorem.
Let $S$ be a connected surface with finite, negative Euler 
characteristic. For simplicity, we will assume that $S$ is orientable,
hence $S$ is an orientable surface of genus $g$ with $n$
open disks removed, so that $\chi(S) = 2-2g-n<0$.
We will construct an explicit compact 3-manifold 
$M$ with toroidal boundary
that satisfies Thurston's conditions and where there exists an
order-two diffeomorphism $R$ with fixed point set a
properly embedded, possibly disconnected, 
separating surface $\mathcal{S}\subset M$,
where one connected component $\S$ of $\mathcal{S}$ is diffeomorphic
to $S$ (and this part of the proof is divided in cases depending
on $g$ and $n$).

Since the interior of $M$ admits a complete hyperbolic
metric of finite volume, Mostow-Prasad rigidity Theorem implies that
$R$ is isotopic to an isometry $\varphi$, which will also be an order-two
diffeomorphism (since $\varphi^2$ is an isometry isotopic to the identity,
hence equal to it).
In particular, the fixed point set of $\varphi$, which is itself isotopic
to $\mathcal{S}$ and contains a component diffeomorphic to $S$, 
is totally geodesic. After producing
this explicit totally geodesic example, we will fix $H\in (0,1)$
and use a property of
the fundamental groups of hyperbolic 3-manifolds of finite volume
to construct finite covers of ${\rm int}(M)$ where $\S$ lifts and
there is a totally umbilic surface $\S_H$, isotopic to the lift
of $\S$ and with constant mean curvature $H$.

\begin{case}
$S$ is an $n$-punctured sphere ($g=0$) for $n\geq 3$.
\end{case}

\noindent 
In this case, let $P = \sn3$ and let $\wh{S}$ be an equatorial sphere
$\sn2$ in $\sn3$. Then, there exists a reflection $R\colon \sn3\to\sn3$
with $\wh{S} = {\rm Fix}(R)$. 
Consider the {\em daisy chain} link $L_{2n}$ in $\sn3$
with $2n$ components, where every other component lies in $\wh{S}$ 
and the other components are {\em perpendicular} to $\wh{S}$ 
in the sense that they are invariant under $R$.
Since $2n\geq 6$, $L_{2n}$ is hyperbolic by Example~\ref{egdaisy}.
In particular, as explained above, the restriction of
$R$ to $N = \sn3\setminus L_{2n}$ is an isometry of the hyperbolic
metric in $N$. The fixed point set of such isometry contains
$n$ components diffeomorphic to $3$-punctured spheres and one component
diffeomorphic to $S$.

\begin{case}\label{casebase1}
$S$ is closed, i.e., $g\geq 2$ and $n=0$.
\end{case}

\noindent
Consider $P = S\times\sn1$ and, after identifying 
$\sn1 = \{(x,y)\in \R^2\mid x^2+y^2 = 1\}$, let
$P^+ = P\cap \{y\geq0\}$, $P^- = P\cap \{y\leq0\}$
and let $R\colon P \to P$ be the reflection that interchanges $P^+$ and
$P^-$, leaving $P^+\cap P^- = S\times\{(-1,0),(1,0)\}$ fixed.
Both $P^+$ and $P^-$ have the topology of
$S\times [0,1]$, so the work of Adams 
{\em et al.}~\cite[Theorem~1.1]{adetal}
imply that there exists a hyperbolic link $L_1$ in $P^+$ such that
$P^+\setminus L_1$ admits a complete hyperbolic metric of finite
volume with totally geodesic boundary. After letting
$L_2 = R(L_1)\subset P^-$ and $L = L_1\cup L_2$, it follows that
the manifold $P\setminus L$ is hyperbolic, the
reflection $R$ restricts to an isometry with fixed point set
equal to $S\times\{(-1,0),(1,0)\}$, providing
two totally geodesic surfaces diffeomorphic to $S$.

\begin{case}\label{egcasebase2}
$n=1$ and $g\geq 2$, i.e., 
$S$ is a one-time-punctured surface of genus $g\geq2$.
\end{case}

\noindent 
Let $S_g$ denote the closed surface of genus $g$ and let
$P = S_g \times \sn1$ and $L$ be as in Case~\ref{casebase1}.
In the hyperbolic metric of $P\setminus L$, let $\gamma^+$
be a minimizing geodesic ray joining a point in $S_g\times\{(-1,0)\}$
to a point (at infinity) in $L_1$. Then, $\gamma^+$ is orthogonal to 
$S_g\times\{(-1,0)\}$ and we may use $R$ to reflect $\gamma^+$
and obtain an arc $\gamma = \gamma^+\cup R(\gamma^+)$
which is a complete geodesic in the hyperbolic metric of 
$P\setminus L$. But a small neighborhood in $P$ of $\gamma$
is a 3-ball $\cB$ which intercepts $L$ in two arcs $g_1\subset L_1$
and $g_2\subset L_2$. We can choose $\cB$ so that $R(\cB) = \cB$
and $R(g_1) = g_2$. Then, we can use the Switch Move
as described previously to replace
the arcs $g_1$ and $g_2$ by a tangle such as in 
Figure~\ref{figaugmented}, and do 
so in an equivariant manner so the trivial circle $C$ added lies
(and bounds a disk in) $S_g\times\{(-1,0)\}$. This creates a new link
$L'$ which is hyperbolic in $P$ and satisfies $R(L') = L'$. 
Furthermore, the reflection $R$ once again restricts to 
an isometry of this hyperbolic metric and the fixed point set
of such isometry has three components: $S_g\times\{(1,0)\}$,
a 3-punctured sphere (more easily seen as a twice-punctured disk) 
bounded by $C$ in $S_g\times\{(-1,0)\}$ and
the other component of $(S_g\times\{(-1,0)\})\setminus C$, which
is diffeomorphic to $S$.

\begin{case}\label{secondtolast}
$g\geq2$ and $n\geq 2$.
\end{case}

\noindent This case once again follows from the previous one.
Just note that in Case~\ref{egcasebase2} we created a hyperbolic
link $L'$ in $P = S_g\times\sn1$ in the ball $\cB$, and the tangle
$\cB\cap L'$ is such that we can apply the Chain Move,
again in an equivariant manner with respect to $R$, and add punctures
to one of the totally geodesic surfaces in 
$(S_g\times\{(-1,0)\})\setminus C$.

\begin{case}
$g = 1$ and $n\geq1$, so $S$ is a torus punctured at least one time.
\end{case}

\noindent Since this case deals with a surface of genus 1 in 
a hyperbolic 3-manifold,
it is probably the most difficult to solve. 
Again by~\cite{adetal}, there exists a link $L_1$ in the interior 
of the compact 3-manifold $P^+ = \bbT^2\times[0,1]$ such that 
$P^+\setminus N(L_1)$ is hyperbolic, but this time, in a distinction
from Case~{casebase1}, the resulting hyperbolic manifold is complete,
so it has no boundary and the Chain Move does not apply directly.
However, it is possible to adapt the proof of the Chain
move, together with the fact that the only obstacle for
hyperbolicity of the manifold 
$M = (\bbT^2\times [-1,1])\setminus N(L_1\cup R(L_1))$, 
where $R\colon \bbT^2\times[-1,1] \to \bbT^2\times[-1,1]$ is
$R(x,t) = (x,-t)$, is the existence of
the essential torus $\bbT^2\times\{0\}$, to prove that
the ad hoc analogous of the Switch move to this specific setting
applies. From here, the proof follows analogously 
as in Cases~\ref{egcasebase2} and~\ref{secondtolast},
firstly obtaining a once-punctured totally geodesic torus
and then adding punctures to it using the Chain move.

\

The arguments in the above cases show that any admissible
orientable topology
for a totally geodesic surface in a hyperbolic 3-manifold
of finite volume actually can be realized as such. Similar
arguments can be done for nonorientable topologies, see~\cite{amr2}.
To finish the sketch of the proof of Theorem~\ref{thmumbilic}, 
we need the following result, which was obtained 
in a series of recent works (see~\cite{3mg} for an appropriate
list of references).

\begin{theorem}
Let $N$ be a noncompact hyperbolic 3-manifold of finite volume.
Then, $\pi_1(N)$ is Locally Extendable Residually Finite (for 
short, LERF), i.e., for every finitely generated subgroup $K$
of $\pi_1(N)$ and any finite set 
$\mathcal{F}\subset \pi_1(N)$, $\mathcal{F}\cap K = \emptyset$, there
exists a representation $\sigma\colon \pi_1(N)\to F$ to
a finite group $F$ such that 
$\sigma(\mathcal{F})\cap \sigma(K) = \emptyset$.
\end{theorem}

The main idea to produce a totally umbilic example from a totally
geodesic one is to use the notion of a {\em $t$-parallel surface}.
Specifically, if $\S$ is a two-sided
totally geodesic surface in a hyperbolic
3-manifold of finite volume $N$ as produced before, and $\S$ is
oriented with respect to a unitary
normal vector field $\eta$, for any $t>0$ we let
$$\Sigma_t = \{\exp_p(t\eta(p))\mid p\in \S\}.$$
Then, $\Sigma_t$ is totally umbilic and has mean 
curvature $H_t = \tanh(t)$. Moreover, it is not difficult to use the
ambient geometry of the ends of $N$ to see that $\Sigma_t$ is proper
and that for small values of $t>0$, $\Sigma_t$ is embedded.
With this in mind, one can look at the first time 
$t_0>0$ where $\S_{t_0}$ is not embedded. Such $t_0$ exists (so
the first intersection point of 
the family $\{\S_{t}\}$ is not at infinity) and,
for some $p\in \S$, the
set $\{\exp_p(s\eta(p))\mid s\in [0,4t_0]\}$ is a closed geodesic
of length $4t_0$, orthogonal to $\S$ in exactly two distinct
points (see~\cite[Lemma~5.2]{amr2}).

The next (and last) step in the proof is to assume that $\tanh(t_0)<H$
(otherwise the theorem follows) and
use that $\pi_1(N)$
is LERF to construct a finite cover $\Pi\colon \wh{N}\to N$ 
where $\Sigma$ lifts to a totally geodesic surface $\wh{\S}$
and there is no closed geodesic in $\wh{N}$ with length 
less than or equal to $5\tanh^{-1}(H)$. In such a manifold,
for $t = \tanh^{-1}(H)$ the corresponding $t$-parallel
surface to $\wh{\S}$, $\wh{\S}_t$, will be a properly embedded
totally umbilic surface as promised.
\end{proof}



\center{Izabella M. de Freitas at izabellamdf@gmail.com\\
Programa de Pós-Graduação em Matemática, Universidade Federal
do Rio Grande do Sul, Porto Alegre, Brazil}
\center{Álvaro K. Ramos at alvaro.ramos@ufrgs.br \\
Departmento de Matemática Pura e Aplicada, Universidade Federal do Rio Grande
do Sul, Porto Alegre, Brazil}


\begin{thebibliography}{99}


\bibitem{Ad1} C. Adams. {\em The Noncompact Hyperbolic 
3-Manifold of Minimal Volume}, Proc. Amer. Math. Soc.
{\bf 100}, no. 4 (1987): 601--606. \url{https://doi.org/10.2307/2046691}.
\bibitem{Adbook}
C. Adams. {\em The knot book}, 
American Mathematical Society, Providence, RI, 2004.
ISBN:0-8218-3678-1, MR:2079925.
\bibitem{adetal} C. Adams, C. Albors-Riera, B. Haddock, 
Z. Li, D. Nishida, B. Reinoso, and L. Wang.
{\em Hyperbolicity of links in thickened surfaces}. 
Topology Appl., {\bf 256}:262--278, (2019). DOI
10.1016/j.topol.2019.01.022, MR3916014.
\bibitem{amr1} C. Adams, W. Meeks III and A. Ramos, {\em Modifications Preserving Hyperbolicity of Link Complements},
to appear in Journal of Algebraic and Geometric Topology, preprint 
available 
at~\href{https://arxiv.org/abs/2007.01739}{arXiv:2007.01739 [math.DG]}
\bibitem{amr2} C. Adams, W. Meeks III and A. Ramos, 
{\em Totally umbilic surfaces in hyperbolic 
3-manifolds of finite volume}. Preprint at \href{https://arxiv.org/abs/2007.03166}{arXiv:2007.03166 [math.DG]}
\bibitem{3mg} M. Aschenbrenner, S. Friedl, H. Wilton. 
{\em 3-manifold groups}, EMS Series of Lectures in Mathematics. 
European Mathematical Society (EMS), Zürich, 2015. xiv+215 pp. 
ISBN: 978-3-03719-154-5.
\bibitem{BPHyp} R. Benedetti and C. Petronio. {\em Lectures on
hyperbolic geometry}. Universitext. Springer-Verlag, Berlin, 1992.
{\rm xiv}+330 pp. ISBN: 3-540-55534-X.
\bibitem{BBetal} L. Bessières, G. Besson, S. Maillot, M. Boileau, and J. Porti. {\em Geometrisation of 3–manifolds}. EMS Tracts in Mathematics, 13.
European Mathematical Society (EMS), Zürich, 2010.
\bibitem{Bonnet} O. Bonnet. {\em Sur quelques propriétés 
des lignes géodésiques}, C.R. Ac. Sc. Paris 40 (1855) 1311-1313.
\bibitem{Bredon} G. Bredon. \emph{Topology and geometry}.
Graduate Texts in Mathematics, 139. Springer-Verlag, New York, 
1993. ISBN: 0-387-97926-3
\bibitem{CaoMeyer} C. Cao and G. Meyerhoff. {\em The orientable 
cusped hyperbolic 3-manifolds of minimum volume}. 
Invent. math. {\bf 146}, 451--478 (2001).
\url{https://doi.org/10.1007/s002220100167}
\bibitem{CaLi} A. Carlotto and C. Li. {\em Constrained deformations 
of positive scalar curvature metrics}. To appear in 
Journal of Differential Geometry, preprint available
at~\href{https://arxiv.org/abs/1903.11772}{arXiv:1903.11772 [math.DG]} 
\bibitem{Snappy} M. Culler, N. Dunfield, M. Goerner, and 
J. Weeks. {\em SnapPy, a computer program for studying the geometry 
and topology of 3-manifolds}, available at 
\url{http://snappy.computop.org}
\bibitem{manfr} M. P. do Carmo. {\em Riemannian geometry}. 
Translated from the second Portuguese edition by Francis Flaherty. 
Mathematics: Theory \& Applications. Birkhäuser Boston, Inc., Boston, 
MA, 1992. {\rm xiv}+300 pp. ISBN: 0-8176-3490-8 MR1138207.
\bibitem{CHR} P. Collin, L. Hauswirth and H. Rosenberg. {\em Minimal 
surfaces in finite volume hyperbolic 3-manifolds $N$ and in 
$M\times \sn1$, $M$ a finite area hyperbolic surface},
Amer. J. Mat., {\bf 140}, 4, (2018) 1075-1112. 
\url{https://doi.org/10.1353/ajm.2018.0024}.
\bibitem{CHMR} P. Collin, L. Hauswirth, L. Mazet and H. Rosenberg.
{\em Minimal surfaces in finite volume noncompact hyperbolic 3-manifolds}.
Trans. Amer. Math. Soc. {\bf 369} (2017), no. 6, 4293--4309.
\bibitem{epstein} D.B.A. Epstein. \emph{Periodic flows on 3-manifolds}, Ann. of Math., 95 (1972), 66-82.
\bibitem{franweeks}
G. Francis and J. Weeks. {\em Conway's ZIP Proof}, 
The American Mathematical Monthly, 106:5, 393-399, (1999).
\bibitem{Giese} H. Gieseking. {\em Analytische untersuchungen 
über topologische gruppen}, PhD Thesis, Münster (1912).
\bibitem{HRW} J. Hass, J. H. Rubinstein and S. Wang. {\em Boundary 
slopes of immersed surfaces in 3-manifolds}. J. Differential Geom. 
{\bf 52} (1999), no. 2, 303--325.
\bibitem{ham1} R. Hamilton. \emph{Three-manifolds with positive Ricci curvature}, J. Differential Geom. 17 (1982), no. 2, 255--306.
\bibitem{ham2} R. Hamilton. \emph{The formation of singularities in the Ricci flow}, Surveys in differential geometry, Vol. II (Cambridge, MA, 1993), 7--136, Int. Press, Cambridge, MA, (1995).
\bibitem{ham3} R. Hamilton. \emph{Non-singular solutions of the Ricci flow on three-manifolds}, Comm. Anal. Geom. 7 (1999), no. 4, 695--729.
\bibitem{Hatcher} A. Hatcher. {\em Notes on Basic 3-Manifold Topology}.
Available at 
\url{https://pi.math.cornell.edu/~hatcher/3M/3Mdownloads.html}
\bibitem{helgason}
S. Helgason. {\em Differential geometry, Lie groups, and 
symmetric spaces}. Corrected reprint of the 1978 original.
Graduate Studies in Mathematics, 34. American Mathematical
Society, Providence, RI, 2001.ISBN: 0-8218-2848-7
\bibitem{JS}
W. Jaco and P. Shalen, {\em Seifert fibered spaces in 3-manifolds},
Mem. Amer. Math. Soc. {\bf 21} (1979), no. 220.
\bibitem{otherJ}
K. Johannson, {\em Homotopy equivalences of 3-manifolds with 
boundaries}, Lecture Notes
in Mathematics, vol. 761, Springer-Verlag, Berlin, (1979).
\bibitem {KleiLott}
B. Kleiner and J. Lott. {\em Notes on Perelman's papers}.
Geom. Topol. {\bf 12}(5): 2587-2855 (2008). 
DOI: 10.2140/gt.2008.12.2587
\bibitem{kneser} H. Kneser. \emph{Geschlossene Flachen in dreidimensionale Mannigfaltigkeiten}, Jahresber. Deutsch. Math.-Verein., 38 (1929), 248--260.
\bibitem{koebe}P. Koebe. \emph{Uber die Uniformisierung beliebiger analytischer Kurven.}Nachr. Ges. Wiss. Gottingen (1907), 191--210 and 633--649.
\bibitem{Liu} G. Liu. {\em 3-manifolds with nonnegative Ricci curvature}. 
Invent. Math. {\bf 193} (2013), no. 2, 367--375.
\bibitem{martelli} B. Martelli. \emph{An Introduction to Geometric Topology}. CreateSpace Independent Publishing Platform (2016). Available
at~\href{https://arxiv.org/abs/1610.02592}{arXiv:1610.02592 [math.GT]} 
\bibitem{MPbook}W. H. Meeks III and J. P\'erez, \emph{Constant mean curvature surfaces in metric Lie groups}. Contemp. Math., \textbf{570} (2012) 25-110.
\bibitem{meramos} W. Meeks III and A. Ramos. {\em Properly 
immersed surfaces in hyperbolic 3-manifolds},
J. Differential Geom., {\bf 112}, 2 (2019), 233-261. 
\url{https://doi.org/10.4310/jdg/1559786424}
\bibitem{milnorsphere} J.W. Milnor. \emph{On manifolds homeomorphic to the 7-sphere}, Ann. of Math. (2), {\bf 64}, N 2 
(1956), 399-405.
\bibitem{milnor} J.W. Milnor. \emph{A unique factorisation theorem for 3-manifolds}, Amer. J. Math., 84 (1962), 1-7.
\bibitem{MorganT1} J. Morgan and G. Tian. {\em Ricci flow and the
Poincaré Conjecture}, Clay Mathematics Monographs, vol. 3, Amer. Math.
Soc., Providence, RI; Clay Mathematics Institute, Cambridge, MA (2007).
\bibitem{MorganT2} J. Morgan and G. Tian. {\em The
Geometrization Conjecture}, 
Clay Mathematics Monographs, vol. 5, Amer. Math.
Soc., Providence, RI; Clay Mathematics Institute, Cambridge, MA (2014).
\bibitem{myers} S. Myers. {\em Riemannian manifolds 
with positive mean curvature}, Duke Math. J. 8 (1941), 401-404.
\bibitem{NeuR} W. Neumann and A. Reid. {\em Arithmetic of hyperbolic
manifolds}. TOPOLOGY 90, Proceedings of the Research Semester in 
Low Dimensional Topology at Ohio State University, 
273--310, (1992).
\bibitem{per1} G. Perelman. \emph{The entropy formula for the Ricci flow and its geometric applications}, preprint (2002), available at
\href{https://arxiv.org/abs/math/0211159}{arXiv:math/0211159 [math.DG]}.
\bibitem{per2} G. Perelman. \emph{Finite extinction time for the solutions to the Ricci flow on certain three-manifolds}, preprint (2003), available
at \href{https://arxiv.org/abs/math/0307245}{arXiv:math/0307245 [math.DG]}.
\bibitem{per3} G. Perelman. \emph{Ricci flow with surgery on three-manifolds}, preprint (2003), available at \href{https://arxiv.org/abs/math/0303109}{arXiv:math/0303109 [math.DG]}.
\bibitem{PoincConj}
H. Poincaré. {\em Cinqui{\`e}me compl{\'e}ment {\`a} l’Analysis situs}.
Rendiconti del Circolo Matematico di Palermo (1884-1940), {\bf 18}, (1904)
45--110.
\bibitem{poinc} H. Poincaré. \emph{Sur l'uniformisation des fonctions analytiques.} Acta Math. 31 (1907), 1--64.
\bibitem{scott} P. Scott. \emph{The geometries of 3-manifolds}, Bull. London Math. Soc. 15 (1983), no. 5, 401--487.
\bibitem{Seifertpaper} H. Seifert. {\em Topologie Dreidimensionaler Gefaserter Räume}, Acta Math. 60 (1933), no. 1, 147--238.
\bibitem{Seifertbook} H. Seifert and W. Threlfall.
{\em Seifert and Threlfall: a textbook of topology}. 
Translated from the German edition of 1934 by Michael A. Goldman. 
With a preface by Joan S. Birman. 
With ``Topology of $3$-dimensional fibered spaces'' by Seifert. 
Translated from the German by Wolfgang Heil. 
Pure and Applied Mathematics, 89. Academic Press, Inc. 
[Harcourt Brace Jovanovich, Publishers], 
New York-London, 1980. {\rm xvi}+437 pp. ISBN: 0-12-634850-2.
\bibitem{thur} W. P. Thurston. \emph{The geometry and topology of 3-manifolds}, Princeton Lecture Notes (1979), available at \url{http://www.msri.org/publications/books/gt3m/}
\bibitem{thur2} W. P. Thurston. \emph{Three-dimensional geometry and topology}. Vol. 1. Edited by Silvio Levy. Princeton Mathematical Series, 35. Princeton University Press, Princeton, NJ, 1997. x+311 pp. ISBN 0-691-08304-5
\bibitem{thur3} W. P. Thurston. \emph{Three dimensional manifolds, Kleinian groups and hyperbolic geometry}, Bull. Amer. Math. Soc., New Ser. 6 (1982), 357--379.
\bibitem{thurletter} W. Thurston. {\em On proof and progress 
in mathematics}, Bull. Amer. Math. Soc. (N.S.),
{\bf 30} (1994), 161--177.

%
%
%
%
%
%
%
\end{thebibliography}
\end{document}